\documentclass[11pt]{amsart}
\usepackage{amssymb,amsfonts,graphicx,color,enumerate,calc}
\usepackage[numbers,sort&compress]{natbib}
\usepackage[colorlinks=true,citecolor=black,linkcolor=black,urlcolor=blue]{hyperref}
\newcommand{\arXiv}[1]{arXiv:\,\href{http://arxiv.org/abs/#1}{#1}}
\newcommand{\MSN}[2]{MR:\,\href{http://www.ams.org/mathscinet-getitem?mr=MR#1}{#1}  (#2)}
\newcommand{\Zbl}[1]{Zbl:\,\href{http://www.zentralblatt-math.org/zmath/en/search/?q=an:#1}{#1}}
\newcommand{\doi}[1]{doi:\,\href{http://dx.doi.org/#1}{#1}}
\newcommand{\half}{\ensuremath{\protect\tfrac{1}{2}}}
\newcommand{\eighth}{\ensuremath{\protect\tfrac{1}{8}}}
\newcommand{\quarter}{\ensuremath{\protect\tfrac{1}{4}}}
\newcommand{\sixth}{\ensuremath{\protect\tfrac{1}{6}}}

\newcommand{\third}{\ensuremath{\protect\tfrac{1}{3}}}
\newcommand{\Oh}[1]{\ensuremath{\protect\mathcal{O}(#1)}}
\theoremstyle{plain} 
\newtheorem{theorem}{Theorem}[section]
\newtheorem{lemma}[theorem]{Lemma}
\newtheorem{corollary}[theorem]{Corollary}
\newtheorem{proposition}[theorem]{Proposition}
\theoremstyle{definition} 

\newcommand{\ceil}[1]{\ensuremath{\protect\lceil#1\rceil}}

\newcommand{\FLOOR}[1]{\ensuremath{\protect\left\lfloor#1\right\rfloor}}
\newcommand{\floor}[1]{\ensuremath{\protect\lfloor#1\rfloor}}
\newcommand{\SET}[1]{\ensuremath{\protect\left\{#1\right\}}}

\newcommand{\thmlabel}[1]{\label{thm:#1}}
\newcommand{\thmref}[1]{Theorem~\ref{thm:#1}}
\newcommand{\twothmref}[2]{Theorems~\ref{thm:#1} and \ref{thm:#2}}
\newcommand{\lemlabel}[1]{\label{lem:#1}}
\newcommand{\lemref}[1]{Lemma~\ref{lem:#1}}
\newcommand{\twolemref}[2]{Lemmas~\ref{lem:#1} and \ref{lem:#2}}
\newcommand{\eqnlabel}[1]{\label{eqn:#1}}
\newcommand{\eqnref}[1]{\eqref{eqn:#1}}
\newcommand{\Eqnref}[1]{Equation~\eqref{eqn:#1}}

\newcommand{\figlabel}[1]{\label{fig:#1}}
\newcommand{\figref}[1]{Figure~\ref{fig:#1}}

\newcommand{\seclabel}[1]{\label{sec:#1}}
\newcommand{\secref}[1]{Section~\ref{sec:#1}}
\newcommand{\twosecref}[2]{Sections~\ref{sec:#1} and \ref{sec:#2}}
\newcommand{\threesecref}[3]{Sections~\ref{sec:#1}, \ref{sec:#2} and \ref{sec:#3}}
\newcommand{\corlabel}[1]{\label{cor:#1}}
\newcommand{\corref}[1]{Corollary~\ref{cor:#1}}
\newcommand{\twocorref}[2]{Corollaries~\ref{cor:#1} and \ref{cor:#2}}
\newcommand{\proplabel}[1]{\label{prop:#1}}
\newcommand{\propref}[1]{Proposition~\ref{prop:#1}}
\newcommand{\twopropref}[2]{Propositions~\ref{prop:#1} and \ref{prop:#2}}
\newcommand{\tablabel}[1]{\label{tab:#1}}
\newcommand{\tabref}[1]{Table~\ref{tab:#1}}

\newcommand{\Figure}[3][htb]{
  \begin{figure}[#1]
    \begin{center}\includegraphics{#2}\end{center}
    \vspace*{-1ex}
    \caption{\figlabel{#2}#3}
  \end{figure}}
\newcommand{\CR}{Chandran and Raju~\citep{CR05,Raju06}}
\newcommand{\A}[1]{\ensuremath{#1^+}}
\newcommand{\B}[1]{\ensuremath{#1^-}}

\newcommand{\CartProd}{\ensuremath{\square}}
\newcommand{\SP}[2]{\ensuremath{#1{\,\boxtimes\,}#2}}
\newcommand{\CP}[2]{\ensuremath{#1{\,\CartProd\,}#2}}
\newcommand{\TCP}[3]{\ensuremath{#1{\,\CartProd\,}#2{\,\CartProd\,}#3}}
\newcommand{\CCP}[3]{\ensuremath{#1{\,\CartProd\,}#2{\,\CartProd}\cdots{\CartProd\,}#3}}
\newcommand{\pw}[1]{\ensuremath{\textup{\textsf{pw}}(#1)}}
\newcommand{\tw}[1]{\ensuremath{\textup{\textsf{tw}}(#1)}}
\newcommand{\bw}[1]{\ensuremath{\textup{\textsf{bw}}(#1)}}
\newcommand{\cd}[1]{\ensuremath{\gamma_{\textup{\textsf{c}}}(#1)}}
\newcommand{\bal}[1]{\ensuremath{\textup{\textsf{bal}}(#1)}}
\newcommand{\hang}[1]{\ensuremath{\textup{\textsf{hang}}(#1)}}
\newcommand{\rad}[1]{\ensuremath{\textup{\textsf{rad}}(#1)}}
\newcommand{\verts}[1]{\ensuremath{\textup{\textsf{v}}(#1)}}
\newcommand{\edges}[1]{\ensuremath{\textup{\textsf{e}}(#1)}}
\newcommand{\STAR}[1]{\ensuremath{\textup{\textsf{star}}(#1)}}
\newcommand{\X}{\ensuremath{\mathcal{X}}}
\title[Clique Minors in Cartesian Products of Graphs]{Clique Minors in \\Cartesian Products of Graphs}

\author{David~R.~Wood}
\address{\newline Department of Mathematics and Statistics
\newline The University of Melbourne
\newline Melbourne, Australia}
\email{woodd@unimelb.edu.au}

\thanks{Supported by QEII Research Fellowship from the Australian Research Council. Research initiated at the Universitat Polit{\`e}cnica de Catalunya (Barcelona, Spain) where supported by a Marie Curie Fellowship of the European Commission under contract MEIF-CT-2006-023865, and by the projects MEC MTM2006-01267 and DURSI 2005SGR00692.}



\begin{document} 
 
\begin{abstract}
  A \emph{clique minor} in a graph $G$ can be thought of as a set of
  connected subgraphs in $G$ that are pairwise disjoint and pairwise
  adjacent. The \emph{Hadwiger number} $\eta(G)$ is the maximum
  cardinality of a clique minor in $G$. It is one of the principle
  measures of the structural complexity of a graph.

This paper studies clique minors in the Cartesian product $G\square H$. Our main result is a rough structural characterisation theorem for Cartesian products with bounded Hadwiger number. It implies that if the product of two sufficiently large graphs has bounded Hadwiger number then it is one of the following graphs:
\begin{itemize}
\item a planar grid with a vortex of bounded width in the outerface,
\item a cylindrical grid with a vortex of bounded width in each of the two `big' faces, or
\item a toroidal grid.
\end{itemize}

Motivation for studying the Hadwiger number of a graph includes Hadwiger's Conjecture, which asserts that the chromatic number $\chi(G)\leq\eta(G)$. It is open whether Hadwiger's Conjecture holds for every Cartesian product. We prove that \CP{G}{H} (where $\chi(G)\geq\chi(H)$) satisfies Hadwiger's Conjecture whenever:
\begin{itemize}
\item $H$ has at least $\chi(G)+1$ vertices, or
\item the treewidth of $G$ is sufficiently large compared to $\chi(G)$.
\end{itemize}
On the other hand, we prove that Hadwiger's Conjecture holds for all Cartesian products if and only if it holds for all $G\square K_2$. We then show that $\eta(G\square K_2)$ is tied to the treewidth of $G$. 

We also develop connections with pseudoachromatic colourings and
connected dominating sets that imply near-tight bounds on the Hadwiger
number of grid graphs (Cartesian products of paths) and Hamming graphs
(Cartesian products of cliques). 
\end{abstract}

\date{November 8, 2007; revised: \today}

\maketitle
\newpage
\tableofcontents
\newpage

\section{Introduction}
\seclabel{Intro}

A \emph{clique minor} in a graph $G$ can be thought of as a set of
connected subgraphs in $G$ that are pairwise disjoint and pairwise
adjacent. The \emph{Hadwiger number} $\eta(G)$ is the maximum
cardinality of a clique minor in $G$. It is one of the principle
measures of the structural complexity of a graph.

Robertson and Seymour \citep{RS-GraphMinorsXVI-JCTB03} proved a rough
structural characterisation of graphs with bounded Hadwiger number. It
says that such a graph can be constructed by a combination of four
ingredients: graphs embedded in a surface of bounded genus, vortices
of bounded width inside a face, the addition of a bounded number of
apex vertices, and the clique-sum operation. Moreover, each of these
ingredients is essential. This result is at the heart of Robertson and
Seymour's proof of Wagner's Conjecture
\citep{RS-GraphMinorsXX-JCTB04}: Every infinite set of finite graphs
contains two graphs, one of which is a minor of the other.

This paper studies clique minors in the (Cartesian) product
\CP{G}{H}. Our main result is a rough structural characterisation of
products with bounded Hadwiger number, which is less rough than the
far more general result by Robertson and Seymour. It says that for
connected graphs $G$ and $H$, each with at least one edge, \CP{G}{H}
has bounded Hadwiger number if and only if at least one of the
following conditions are satisfied:
\begin{itemize}
\item $G$ has bounded treewidth and $H$ has bounded order,
\item $H$ has bounded treewidth and $G$ has bounded order, or
\item $G$ has bounded hangover and $H$ has bounded hangover,
\end{itemize}
where hangover is a parameter defined in \secref{RoughGraphs}. Basically, a graph with bounded hangover is
either a cycle or consists of a path of degree-2 vertices joining two
connected subgraphs of bounded order with no edge between the
subgraphs. This implies that if the product of two sufficiently large
graphs has bounded Hadwiger number then it is one of the following
graphs:
\begin{itemize}
\item a planar grid (the product of two paths) with a vortex of
  bounded width in the outerface,
\item a cylindrical grid (the product of a path and a cycle) with a
  vortex of bounded width in each of the two `big' faces, or
\item a toroidal grid (the product of two cycles).
\end{itemize}
The key case for the proof of this structure theorem is when $G$ and
$H$ are trees. This case is handled in \secref{TwoTrees}. The proof
for general graphs is given in
\twosecref{GeneralComplete}{RoughGraphs}.

Before proving our main results we develop connections with
pseudoachromatic colourings (\secref{Pseudo}) and connected dominating
sets (\twosecref{StarMinors}{Domination}) that imply near-tight bounds
on the Hadwiger number of grid graphs (products of paths;
\threesecref{Grids}{Pseudo}{Domination}) and Hamming graphs (products
of cliques; \secref{Hamming}). As summarised in \tabref{Results}, in
each case, we improve the best previously known lower bound by a
factor of between $\Omega(n^{1/2})$ and $\Omega(n^{3/2})$ to conclude
asymptotically tight bounds for fixed $d$.

\begin{table}[hbt]
  \tablabel{Results}
  \caption{Improved lower bounds on the Hadwiger number of specific graphs.}
  \begin{tabular}{ccccc}
    \hline
    graph	& $d$	& previous best	& new result 	& reference \\\hline
    grid graph $P_n^d$	&
    even	&
    $\Omega(n^{(d-2)/2})$ &
    $\Theta(n^{d/2})$	& 
    \thmref{EvenGrid}\\
    grid graph $P_n^d$	&
    odd	&
    $\Omega(n^{(d-1)/2})$ &
    $\Theta(n^{d/2})$	& 
    \thmref{OddDimGrid}\\
    Hamming graph $K_n^d$	&
    even	&
    $\Omega(n^{(d-2)/2})$	& 
    $\Theta(n^{(d+1)/2})$	&
    \thmref{CompleteProduct}\\
    Hamming graph $K_n^d$	&
    odd	& 
    $\Omega(n^{(d-1)/2})$	& 
    $\Theta(n^{(d+1)/2})$	&
    \thmref{CompleteProduct}\\\hline
  \end{tabular}
\end{table}

\subsection{Hadwiger's Conjecture}

Motivation for studying clique minors includes Hadwiger's Conjecture,
a far reaching generalisation of the $4$-colour theorem, which states
that the chromatic number $\chi(G)\leq\eta(G)$ for every graph $G$. It
is open whether Hadwiger's Conjecture holds for every product. The
following classes of products are known to satisfy Hadwiger's
Conjecture (where $G$ and $H$ are connected and $\chi(G)\geq\chi(H)$):
\begin{itemize}
\item The product of sufficiently many graphs relative to their
  maximum chromatic number satisfies Hadwiger's Conjecture
  \citep{CS-DM07}.
\item If $\chi(H)$ is not too small relative to $\chi(G)$, then
  \CP{G}{H} satisfies Hadwiger's Conjecture \citep{CR05,Raju06}.
\end{itemize}
See \secref{HadwigerConjecture} for precise versions of this
statements. We add to this list as follows:
\begin{itemize}
\item If $H$ has at least $\chi(G)+1$ vertices, then \CP{G}{H}
  satisfies Hadwiger's Conjecture (\thmref{HadwigerBigH}).
\item If the treewidth of $G$ is sufficiently large compared to
  $\chi(G)$, then \CP{G}{H} satisfies Hadwiger's Conjecture
  (\thmref{HadwigerBigTW}).
\end{itemize}
On the other hand, we prove that Hadwiger's Conjecture holds for all
\CP{G}{H} with $\chi(G)\geq\chi(H)$ if and only if Hadwiger's
Conjecture holds for \CP{G}{K_2}. We then show that
$\eta(\CP{G}{K_2})$ is tied to the treewidth of $G$. All these results
are presented in \secref{HadwigerConjecture}.

Clique minors in products have been previously considered by a number
of authors \citep{Zelinka-MS76, Miller-DM78, Ivanco88, ABPS97,
  Kotlov-EuJC01, CR05, Raju06, CS-DM07}. In related work,
\citet{XuYang-DM06} and \citet{Spacapan08} studied the connectivity of
products, \citet{DS-RSA06} studied minors in lifts of graphs, and
\citet{Goldberg-LAA09} studied the Colin de Verdi\`ere  number of
products. See \citep{IKR,GraphProducts,ProductsHandbook} for more on graph products. 

\section{Preliminaries}

All graphs considered in this paper are undirected, simple, and
finite; see \citep{Bollobas,Diestel00}. Let $G$ be a graph with vertex
$V(G)$ and edge set $E(G)$. Let $\verts{G}=|V(G)|$ and
$\edges{G}=|E(G)|$ respectively denote the \emph{order} and
\emph{size} of $G$. Let $\Delta(G)$ denote the maximum degree of $G$.
The  \emph{chromatic number} of $G$, denoted by $\chi(G)$, is the
  minimum integer $k$ such that each vertex of $G$ can be assigned one
  of $k$ colours such that adjacent vertices receive distinct
  colours. Let $K_n$ be the complete graph with $n$ vertices. A \emph{clique} of
a graph $G$ is a complete subgraph of $G$. The \emph{clique number} of
$G$, denoted by $\omega(G)$, is the maximum order of a clique of
$G$. Let $P_n$ be the path with $n$ vertices. By default, $V(K_n)=[n]$
and $P_n=(1,2,\dots,n)$. A \emph{leaf} in a graph is a vertex of
degree $1$. Let $S_n$ be the star graph with $n$ leaves; that is,
$S_n=K_{1,n}$.

\subsection{Vortices}

Consider a graph $H$ embedded in a surface; see \citep{MoharThom}. Let
$(v_1,v_2,\dots,v_k)$ be a facial cycle in $H$. Consider a graph $G$
obtained from $H$ by adding sets of vertices $S_1,S_2,\dots,S_k$
(called \emph{bags}), such that for each $i\in[k]$ we have $v_i\in
S_i\cap V(H)\subseteq\{v_1,\dots,v_k\}$, and for each vertex $v\in
\cup_iS_i$, if $R(v):=\{i\in[k],v\in S_i\}$ then for some $i,j$,
either $R(v)=[i,j]$ or $R(v)=[i,k]\cup[1,j]$, and for each edge $vw\in
E(G)$ with $v,w\in\cup_i S_i$ there is some $i\in[k]$ for which
$v,w\in S_i$. Then $G$ is obtained from $H$ by \emph{adding a vortex
  of width} $\max_i|S_i|$.

\subsection{Graph Products}

Let $G$ and $H$ be graphs. The \emph{Cartesian} (or \emph{square})
\emph{product} of $G$ and $H$, denoted by \CP{G}{H}, is the graph with
vertex set $$V(\CP{G}{H})\,:=\,V(G)\times V(H)\,:=\,\{(v,x):v\in
V(G),x\in V(H)\}\enspace,$$ where $(v,x)(w,y)$ is an edge of \CP{G}{H}
if and only if $vw\in E(G)$ and $x=y$, or $v=w$ and $xy\in E(H)$.

Assuming isomorphic graphs are equal, the Cartesian product is
commutative and associative, and $\CCP{G_1}{G_2}{G_d}$ is
well-defined. We can consider a Cartesian product
$G:=\CCP{G_1}{G_2}{G_d}$ to have vertex set
$$V(G)=\{v=(v_1,v_2,\dots,v_d):v_i\in V(G_i),i\in[d]\}\enspace,$$
where $vw\in E(G)$ if and only if $v_iw_i\in E(G_i)$ for some $i$, and
$v_j=w_j$ for all $j\ne i$; we say that the edge $vw$ is in
\emph{dimension} $i$.  For a graph $G$ and integer $d\geq1$, let $G^d$
denote the $d$-fold Cartesian product 
$$G^d:=\underbrace{G\,\square\, G\,\square\,\cdots\,\square\, G}_d\enspace.$$

Since the Cartesian product is the focus of this paper, it will
henceforth be simply referred to as the \emph{product}. Other graph
products will be briefly discussed. The \emph{direct product} $G\times
H$ has vertex set $V(G)\times V(H)$, where $(v,x)$ is adjacent to
$(w,y)$ if and only if $vw\in E(G)$ and $xy\in E(H)$. The \emph{strong
  product} \SP{G}{H} is the union of \CP{G}{H} and $G\times H$. The
\emph{lexicographic product} (or \emph{graph composition}) $G\cdot H$
has vertex set $V(G)\times V(H)$, where $(v,x)$ is adjacent to $(w,y)$
if and only if $vw\in E(G)$, or $v=w$ and $xy\in E(H)$. Think of
$G\cdot H$ as being constructed from $G$ by replacing each vertex of
$G$ by a copy of $H$, and replacing each edge of $G$ by a complete
bipartite graph. Note that the lexicographic product is not
commutative.

\subsection{Graph Minors}

A graph $H$ is a \emph{minor} of a graph $G$ if $H$ can be obtained
from a subgraph of $G$ by contracting edges. For each vertex $v$ of
$H$, the connected subgraph of $G$ that is contracted into $v$ is
called a \emph{branch set} of $H$.  Two subgraphs $X$ and $Y$ in $G$
are \emph{adjacent} if there is an edge with one endpoint in $X$ and
the other endpoint in $Y$. A $K_n$-minor of $G$ is called a
\emph{clique minor}. It can be thought of as $n$ connected subgraphs
$X_1,\dots,X_n$ of $G$, such that distinct $X_i$ and $X_j$ are
disjoint and adjacent. The \emph{Hadwiger number} of $G$, denoted by
$\eta(G)$, is the maximum $n$ such that $K_n$ is a minor of $G$.

The following observation is used repeatedly.

\begin{lemma}
  \lemlabel{TotalMinor} If $H$ is a minor of a connected graph $G$,
  then $G$ has an $H$-minor such that every vertex of $G$ is in some
  branch set.
\end{lemma}

\begin{proof}
  Start with an $H$-minor of $G$. If some vertex of $G$ is not in a
  branch set, then since $G$ is connected, some vertex $v$ of $G$ is
  not in a branch set and is adjacent to a vertex that is in a branch
  set $X$. Adding $v$ to $X$ gives an $H$-minor using more vertices of
  $G$. Repeat until every vertex of $G$ is in some branch set.
\end{proof}

In order to describe the principal result of this paper
(\thmref{RoughGraph}), we introduce the following formalism. Let
$\alpha:\X\rightarrow\mathbb{R}$ and $\beta:\X\rightarrow\mathbb{R}$
be functions, for some set \X. Then $\alpha$ and $\beta$ are
\emph{tied} if there is a function $f$ such that $\alpha(x)\leq
f(\beta(x))$ and $\beta(x)\leq f(\alpha(x))$ for all
$x\in\mathcal{X}$. \thmref{RoughGraph} presents a function that is
tied to $\eta(\CP{G}{H})$.

\subsection{Upper Bounds on the Hadwiger Number}

To prove the tightness of our lower bound constructions, we use the
following elementary upper bounds on the Hadwiger number.

\begin{lemma}
  \lemlabel{UpperBound} For every connected graph $G$ with average
  degree at most $\delta\geq2$,
  \begin{align*}
    \eta(G)\leq\sqrt{(\delta-2)\verts{G}}+3\enspace.
  \end{align*}
\end{lemma}

\begin{proof}
  Let $k:=\eta(G)$. Say $X_1,\dots,X_k$ are the branch sets of
  $K_k$-minor in $G$. By \lemref{TotalMinor}, we may assume that every
  vertex is in some branch set. Since at least $\binom{k}{2}$ edges
  have endpoints in distinct branch sets,
  \begin{equation*}
    \eqnlabel{UpperBound}
    \edges{G}
    \geq\binom{k}{2}+\sum_{i=1}^k\edges{X_i}
    \geq\binom{k}{2}+\sum_{i=1}^k\big(\verts{X_i}-1\big)
    =\binom{k}{2}+\verts{G}-k.
  \end{equation*}
  Since $2\edges{G}=\delta\,\verts{G}$, we have
  $k^2-3k-(\delta-2)\verts{G}\leq 0$. The result follows from the
  quadratic formula.
\end{proof}

The following result, first proved by \citet{Ivanco88}, is another
elementary upper bound on $\eta(G)$. It is tight for a surprisingly
large class of graphs; see \propref{LexCliqueMinor}. We include the
proof for completeness.

\begin{lemma}[\citep{Ivanco88,Stiebitz-DM92}]
  \lemlabel{NewUpperBound} For every graph $G$,
  \begin{align*}
    \eta(G)\leq \FLOOR{\frac{\verts{G}+\omega(G)}{2}} \enspace.
  \end{align*}
\end{lemma}

\begin{proof}
  Consider a $K_n$-minor in $G$, where $n:=\eta(G)$. For $j\geq1$, let
  $n_j$ be the number of branch sets that contain exactly $j$
  vertices. Thus $$\verts{G}-n_1\geq\sum_{j\geq2}j\cdot
  n_j\geq2\sum_{j\geq2}n_j=2(n-n_1).$$ Hence $\verts{G}+n_1\geq
  2n$. The branch sets that contain exactly one vertex form a
  clique. Thus $n_1\leq \omega(G)$ and $\verts{G}+\omega(G)\geq
  2n$. The result follows.
\end{proof}

\subsection{Treewidth, Pathwidth and Bandwidth}

Another upper bound on the Hadwiger number is obtained as follows.  A
\emph{tree decomposition} of a graph $G$ consists of a tree $T$ and a
set $\{T_x\subseteq V(G):x\in V(T)\}$ of `bags' of vertices of $G$
indexed by $T$, such that
\begin{itemize}
\item for each edge $vw\in E(G)$, there is some bag $T_x$ that
  contains both $v$ and $w$, and
\item for each vertex $v\in V(G)$, the set $\{x\in V(T):v\in T_x\}$
  induces a non-empty (connected) subtree of $T$.
\end{itemize}
The \emph{width} of the tree decomposition is $\max\{|T_x|:x\in
V(T)\}-1$. The \emph{treewidth} of $G$, denoted by \tw{G}, is the
minimum width of a tree decomposition of $G$. For example, $G$ has
treewidth $1$ if and only if $G$ is a forest. A tree decomposition
whose underlying tree is a path is called a \emph{path decomposition},
and the \emph{pathwidth} of $G$, denoted by \pw{G}, is the minimum
width of a path decomposition of $G$. The \emph{bandwidth} of $G$,
denoted by \bw{G}, is the minimum, taken over of all linear orderings
$(v_1,\dots,v_n)$ of $V(G)$, of $\max\{|i-j|:v_iv_j\in E(G)\}$. It is
well known \citep{Bodlaender-TCS98} that for every graph $G$,
\begin{equation}
  \eqnlabel{TreePathBand}
  \eta(G)\leq\tw{G}+1\leq\pw{G}+1\leq\bw{G}+1\enspace.
\end{equation}

\section{Hadwiger Number of Grid Graphs} \seclabel{Grids}

In this section we consider the Hadwiger number of the products of
paths, so called \emph{grid} graphs. First consider the $n\times m$
grid \CP{P_n}{P_m}. It has no $K_5$-minor since it is planar. In fact,
$\eta(\CP{P_n}{P_m})=4$ for all $n\geq m\geq3$. Similarly,
\CP{P_n}{P_2} has no $K_4$-minor since it is outerplanar, and
$\eta(\CP{P_n}{P_2})=3$ for all $n\geq2$.

Now consider the \emph{double-grid} \TCP{P_n}{P_m}{P_2}, where $n\geq
m\geq 2$.  For $i\in[n]$, let $C_i$ be the $i$-th column in the base
copy of \CP{P_n}{P_m}; that is, $C_i:=\{(i,y,1):y\in[m]\}$.  For
$j\in[m]$, let $R_j$ be the $j$-th row in the top copy of
\CP{P_n}{P_m}; that is, $R_j:=\{(x,j,2):x\in[n]\}$.  Since each $R_i$
and each $C_j$ are adjacent, contracting each $R_i$ and each $C_j$
gives a $K_{n,m}$-minor. \citet{CS-DM07} studied the case $n=m$, and
observed that a $K_m$-minor is obtained by contracting a matching of
$m$ edges in $K_{m,m}$. In fact, contracting a matching of $m-1$ edges
in $K_{n,m}$ gives a $K_{m+1}$-minor. (In fact, $\eta(K_{m,m})=m+1$;
see \citep{Wood-GC07} for example.)\ Now observe that $R_1$ is
adjacent to $R_2$ and $C_1$ is adjacent to $C_2$. Thus contracting
each edge of the matching $R_3C_3,R_4C_4,\dots,R_mC_m$ gives a
$K_{m+2}$-minor, as illustrated in \figref{DoubleGrid}. Hence
$\eta(\TCP{P_n}{P_m}{P_2})\geq m+2$.

\Figure{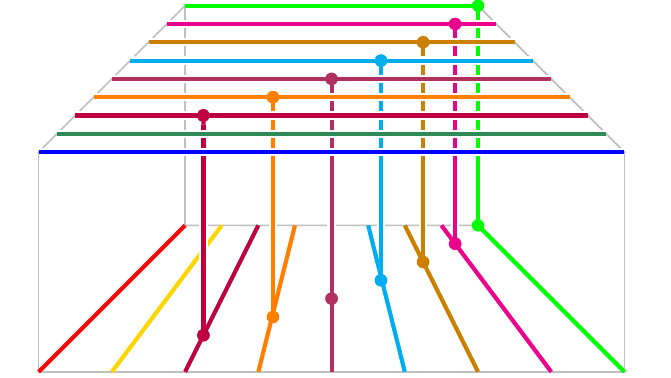}{A $K_{m+2}$-minor in \TCP{P_m}{P_m}{P_2}.}

Now we prove a simple upper bound on $\eta(\TCP{P_n}{P_m}{P_2})$.
Clearly, $$\bw{\CP{P_n}{P_m}}\leq m\enspace.$$ Thus, by
\lemref{BandwidthProduct} and \eqnref{TreePathBand},
$$\eta(\TCP{P_n}{P_m}{P_2})\leq \bw{\TCP{P_n}{P_m}{P_2}}+1\leq 2m+1\enspace.$$
Summarising\footnote{In the case $n=m$, \citet{CS-DM07} claimed an
  upper bound of $\eta(\TCP{P_m}{P_m}{P_2})\leq 2m+2$ without proof.},
\begin{equation}
  \eqnlabel{DoubleGrid}
  m+2\leq \eta(\TCP{P_n}{P_m}{P_2})\leq 2m+1.
\end{equation}
We conjecture that the lower bound in \eqnref{DoubleGrid} is the
answer; that is, $$\eta(\TCP{P_n}{P_m}{P_2})=m+2\enspace.$$


The above construction of a clique minor in the double-grid
generalises as follows.

\begin{proposition}
  \proplabel{DoubleGridLike} For all connected graphs $G$ and $H$,
  each with at least one edge,
  \begin{align*}
    \eta(\TCP{G}{H}{P_2}) &\geq
    \omega(G)+\omega(H)+\min\{\verts{G}-\omega(G),\verts{H}-\omega(H)\}\\
    &\geq
    \min\{\verts{G},\verts{H}\}+\min\{\omega(G),\omega(H)\}\\
    &\geq \min\{\verts{G},\verts{H}\}+2\enspace.
  \end{align*}
\end{proposition}

\begin{proof}
  Let $P$ be a maximum clique of $G$.  Let $Q$ be a maximum clique of
  $H$.  Without loss of generality, $n:=\verts{G}-\omega(G)\leq
  \verts{H}-\omega(H)$.  Say $V(G)-V(P)=\{v_1,v_2,\dots,v_n\}$ and
  $V(H)-V(Q)=\{w_1,w_2,\dots,w_m\}$, where $n\leq m$. Let
  $V(P_2)=\{1,2\}$.

  For $x\in V(G)$, let $A\langle x\rangle$ be the subgraph of
  \TCP{G}{H}{P_2} induced by $\{(x,y,1):y\in V(H)\}$.  For $y\in
  V(G)$, let $B\langle y\rangle$ be the subgraph of \TCP{G}{H}{P_2}
  induced by $\{(x,y,2):x\in V(G)\}$.  Note that each subgraph
  $A\langle x\rangle$ is isomorphic to $H$, and is thus
  connected. Similarly, each subgraph $B\langle y\rangle$ is
  isomorphic to $G$, and is thus connected.

  Distinct subgraphs $A\langle x\rangle$ and $A\langle x'\rangle$ are
  disjoint since the first coordinate of every vertex in $A\langle
  x\rangle$ is $x$.  Distinct subgraphs $B\langle y\rangle$ and
  $B\langle y'\rangle$ are disjoint since the second coordinate of
  every vertex in $B\langle x\rangle$ is $y$.  Subgraphs $A\langle
  x\rangle$ and $B\langle y\rangle$ are disjoint since the third
  coordinate of every vertex in $A\langle x\rangle$ is $1$, and the
  third coordinate of every vertex in $B\langle y\rangle$ is $2$.

  Since the vertex $(x,y,1)$ in $A\langle x\rangle$ is adjacent to the
  vertex $(x,y,1)$ in $B\langle y\rangle$, the $A\langle x\rangle$ and
  $B\langle y\rangle$ subgraphs are the branch sets of a complete
  bipartite $K_{\verts{G},\verts{H}}$-minor in
  \TCP{G}{H}{H}. Moreover, for distinct vertices $x$ and $x'$ in the
  clique $P$, for any vertex $y\in V(H)$, the vertex $(x,y,1)$ in
  $A\langle x\rangle$ is adjacent to the vertex $(x',y,1)$ in
  $A\langle x'\rangle$. Similarly, for distinct vertices $y$ and $y'$
  in the clique $Q$, for any vertex $x\in V(G)$, the vertex $(x,y,2)$
  in $B\langle y\rangle$ is adjacent to the vertex $(x,y',2)$ in
  $B\langle y'\rangle$. For each $i\in[n]$, let $X_i$ be the subgraph
  induced by $A\langle v_i\rangle \cup B\langle w_i\rangle$. Now
  $X\langle i\rangle$ is connected, since the vertex $(v_i,w_i,1)$ in
  $A\langle i\rangle$ is adjacent to the vertex $(v_i,w_i,2)$ in
  $B\langle i\rangle$.

  We have shown that $\{A\langle x\rangle:x\in P\}\cup \{B\langle
  x\rangle:x\in Q\}\cup\{X_i:i\in [n]\}$ is a set of
  $\omega(G)+\omega(H)+n$ connected subgraphs, each pair of which are
  disjoint and adjacent. Hence these subgraphs are the branch sets of
  a clique minor in \TCP{G}{H}{P_2}. Therefore
  $\eta(\TCP{G}{H}{P_2})\geq \omega(G)+\omega(H)+n$, as desired. The
  final claims are easily verified.
\end{proof}


Now consider the Hadwiger number of the $d$-dimensional grid graph
$$P_n^d:=\underbrace{\CCP{P_n}{P_n}{P_n}}_d\enspace.$$
The best previously known bounds are due to \citet{CS-DM07} who proved
that $$n^{\floor{(d-1)/2}}\leq
\eta(P_n^d)\leq\sqrt{2d}\,n^{d/2}+1\enspace.$$ In the case that $d$ is
even we now improve this lower bound by a $\Theta(n)$ factor, and thus
determine $\eta(P_n^d)$ to within a factor of $4\sqrt{2d}$ (ignoring
lower order terms).

\begin{theorem}
  \thmlabel{EvenGrid} For every integer $n\geq2$ and even integer
  $d\geq4$,
$$\tfrac{1}{4}n^{d/2}-\Oh{n^{d/2-1}}
\leq \eta(P_n^d)< \sqrt{2d-2}\,n^{d/2}+3\enspace.$$
\end{theorem}

\begin{proof}
  The upper bound follows from \lemref{UpperBound} since
  $\verts{P_n^d}=n^d$ and $\Delta(P_n^d)=2d$.  Now we prove the lower
  bound. Let $V(P_n^d)=[n]^d$, where two vertices are adjacent if and
  only if they share $d-1$ coordinates in common and differ by $1$ in
  the remaining coordinate. Let $p:=\frac{d}{2}$.

  For $j_1,j_2\in\floor{\tfrac{n}{2}}$ and $j_3,\dots,j_p\in[n]$, let
  $A\langle j_1,\dots,j_p\rangle$ be the subgraph of $P_n^d$ induced
  by
  \begin{align*}
    \{(2j_1,2j_2,j_3,j_4,\dots,j_p,x_1,x_2,\dots,x_p):x_i\in[n],i\in[p]\};
  \end{align*}
  let $B\langle j_1,\dots,j_p\rangle$ be the subgraph of $P_n^d$
  induced by
  \begin{align*}
    &\{(2x_1-1,x_2,x_3,\dots,x_p,j_1,j_2,\dots,j_p)
    :x_1\in\floor{\tfrac{n}{2}},x_i\in[n],i\in[2,p]\}\\
    \cup &\{(x_1,2x_2-1,x_3,x_4\dots,x_p,j_1,j_2,\dots,j_p)
    :x_1\in[n],x_2\in\floor{\tfrac{n}{2}},\\
    &\hspace*{70mm}x_i\in[n],i\in[3,p]\};
  \end{align*}
  and let $X\langle j_1,\dots,j_p\rangle$ be the subgraph of $P_n^d$
  induced by $A\langle j_1,\dots,j_p\rangle\cup B\langle
  j_1,\dots,j_p\rangle$.

  Each $A\langle j_1,\dots,j_p\rangle$ subgraph is disjoint from each
  $B\langle j'_1,\dots,j'_p\rangle$ subgraph since the first and
  second coordinates of each vertex in $A\langle j_1,\dots,j_p\rangle$
  are both even, while the first or second coordinate of each vertex
  in $B\langle j'_1,\dots,j'_p\rangle$ is odd. Two distinct subgraphs
  $A\langle j_1,\dots,j_p\rangle$ and $A\langle
  j'_1,\dots,j'_p\rangle$ are disjoint since the $p$-tuples determined
  by the first $p$ coordinates are distinct. Similarly, two distinct
  subgraphs $B\langle j_1,\dots,j_p\rangle$ and $B\langle
  j'_1,\dots,j'_p\rangle$ are disjoint since the $p$-tuples determined
  by the last $p$ coordinates are distinct. Hence each pair of
  distinct subgraphs $X\langle j_1,\dots,j_p\rangle$ and $X\langle
  j'_1,\dots,j'_p\rangle$ are disjoint.

  Observe that $A\langle j_1,\dots,j_p\rangle$ is isomorphic to
  $P^p_n$, and is thus connected. In particular, every pair of
  vertices $(2j_1,2j_2,j_3,j_4,\dots,j_p,x_1,x_2,\dots,x_p)$ and
  $(2j_1,2j_2,j_3,j_4,\dots,j_p,x'_1,x'_2,\dots,x'_p)$ in $A\langle
  j_1,\dots,j_p\rangle$ are connected by a path of length
  $\sum_i|x_i-x'_i|$ contained in $A\langle j_1,\dots,j_p\rangle$. To
  prove that $B\langle j_1,\dots,j_p\rangle$ is connected, consider a
  pair of vertices $$v=(x_1,x_2,\dots,x_p,j_1,j_2,\dots,j_p)\text{ and
  } v'=(x'_1,x'_2,\dots,x'_p,j_1,j_2,\dots,j_p)$$ in $B\langle
  j_1,\dots,j_p\rangle$.  If $x_1$ is even then walk along any one of
  the dimension-$1$ edges incident to $v$. This neighbour is in
  $B\langle j_1,\dots,j_p\rangle$, and its first coordinate is
  odd. Thus we can now assume that $x_1$ is odd. Similarly, we can
  assume that $x_2$, $x'_1$, and $x'_2$ are all
  odd. Then $$(x_1,x_2,\dots,x_p,j_1,j_2,\dots,j_p)\text{ and
  }(x'_1,x'_2,\dots,x'_p,j_1,j_2,\dots,j_p)$$ are connected by a path
  of length $\sum_i|x_i-x'_i|$ contained in $B\langle
  j_1,\dots,j_p\rangle$.  Thus $B\langle j_1,\dots,j_p\rangle$ is
  connected. The vertex
$$(2j_1,2j_2,j_3,j_4,\dots,j_p,j_1,j_2,\dots,j_p)$$ 
in $A\langle j_1,\dots,j_p\rangle$ is adjacent to the vertex
$$(2j_1-1,2j_2,j_3,j_4,\dots,j_p,j_1,j_2,\dots,j_p)$$ 
in $B\langle j_1,\dots,j_p\rangle$. Thus $X\langle
j_1,\dots,j_p\rangle$ is connected.  Each pair of subgraphs $X\langle
j_1,\dots,j_p\rangle$ and $X\langle j'_1,\dots,j'_p\rangle$ are
adjacent since the
vertex $$(2j_1,2j_2,j_3,j_4,\dots,j_p,j'_1,j'_2,\dots,j'_p)$$ in
$A\langle j_1,\dots,j_p\rangle$ is adjacent to the
vertex $$(2j_1-1,2j_2,j_3,j_4,\dots,j_p,j'_1,j'_2,\dots,j'_p)$$ in
$B\langle j'_1,\dots,j'_p\rangle$.

Hence the $X\langle j_1,\dots,j_p\rangle$ form a complete graph minor
in $P_n^d$ of order
$n^{d/2-2}\floor{\tfrac{n}{2}}^2=\tfrac{1}{4}n^{d/2}-\Oh{n^{d/2-1}}$.
\end{proof}

Note that for particular values of $n$, the lower bound in
\thmref{EvenGrid} is improved by a constant factor in
\corref{SquareGrid} below.

\section{Odd-Dimensional Grids and Pseudoachromatic Colourings}
\seclabel{Pseudo}

The `dimension pairing' technique used in \secref{Grids} to construct
large clique minors in even-dimensional grids does not give tight
bounds for odd-dimensional grids. To construct large clique minors in
odd-dimensional grids we use the following idea.

A \emph{pseudoachromatic $k$-colouring} of a graph $G$ is a function
$f:V(G)\rightarrow[k]$ such that for all distinct $i,j\in[k]$ there is
an edge $vw\in E(G)$ with $f(v)=i$ and $f(w)=j$. The
\emph{pseudoachromatic number} of $G$, denoted by $\psi(G)$, is the
maximum integer $k$ such that there is a pseudoachromatic
$k$-colouring of $G$. Pseudoachromatic colourings were introduced by
\citet{Gupta69} in 1969, and have since been widely studied.
For example, many authors \citep{Yeg-SEABM00, MR-DM01, HM-DM76,
  GK-FM74} have proved\footnote{For completeness, we prove that
  $\psi(P_n)>\sqrt{2n-2}-2$. Let $P_n=(x_1,\dots,x_n)$.  Let $t$ be
  the maximum odd integer such that $\binom{t}{2}\leq n-1$.  Then
  $t>\sqrt{2n-2}-2$.  Denote $V(K_t)$ by $\{v_1,\dots,v_t\}$.  Since
  $t$ is odd, $K_t$ is Eulerian.  Orient the edges of $K_t$ by
  following an Eulerian cycle $C=(e_1,e_2,\dots,e_{\binom{t}{2}})$.
  For $\ell\in\binom{t}{2}$, let $f(x_\ell)=i$, where
  $e_\ell=(v_i,v_j)$.  For $\ell\in[\binom{t}{2}+1,n]$, let
  $f(x_\ell)=1$.  Consider distinct colours $i,j\in[t]$.  Thus for
  some edge $e_\ell$ of $K_t$, without loss of generality,
  $e_\ell=(v_i,v_j)$.  Say $e_{\ell+1}=(v_j,v_k)$ is the next edge in
  $C$, where $e_{\ell+1}$ means $e_1$ if $\ell=\binom{t}{2}$.  Since
  $\ell\leq\binom{t}{2}\leq n-1$, we have $\ell+1\in[n]$.  By
  construction, $f(x_\ell)=i$ and $f(x_{\ell+1})=j$.  Thus $f$ is a
  pseudoachromatic colouring.}  that
\begin{equation}
  \eqnlabel{PseudoPath}
  \psi(P_n)>\sqrt{2n-2}-2.
\end{equation}

Note that the only difference between a pseudoachromatic colouring and
a clique minor is that each colour class is not necessarily
connected. We now show that the colour classes in a pseudoachromatic
colouring can be made connected in a three-dimensional product.

\begin{theorem}
  \thmlabel{PseudoAchromatic} Let $G$, $H$ and $I$ be connected
  graphs.  Let $A$, $B$ and $C$ be connected minors of $G$, $H$ and $I$ respectively,
  such that each branch set in each minor has at least two vertices. 
If $\verts{B}\geq \verts{C}$  then
$$\eta(\TCP{G}{H}{I})\geq \min\{\psi(A),\verts{B}\}\cdot \verts{C}\enspace.$$
\end{theorem}

\begin{proof}
  (The reader should keep the example of $G=H=I=P_n$ and
  $A=B=C=P_{\floor{n/2}}$ in mind.)\
 
  Let $V(B)=\{y_1,\dots,y_{\verts{B}}\}$ and
  $V(C)=\{z_1,\dots,z_{\verts{C}}\}$. By contracting edges in $H$, we
  may assume that there are exactly two vertices of $H$ in each branch
  set of $B$. Label the two vertices of $H$ in the branch set
  corresponding to each $y_j$ by \A{y_j} and \B{y_j}. By contracting
  edges in $I$, we may assume that there are exactly two vertices of
  $I$ in each branch set of $C$. Label the two vertices of $I$ in the
  branch set corresponding to each $z_j$ by \A{z_j} and \B{z_j}.

  Let $k:=\min\{\psi(A),\verts{B}\}$. Let $f:V(A)\rightarrow[k]$ be a
  pseudoachromatic colouring of $A$. Our goal is to prove that
  $\eta(\TCP{G}{H}{I})\geq k\cdot \verts{C}$.

  By contracting edges in $G$, we may assume that there are exactly
  two vertices of $G$ in each branch set of $A$. Now label the two
  vertices of $G$ in the branch set corresponding to each vertex $v$
  of $A$ by \A{v} and \B{v} as follows. Let $T$ be a spanning tree of
  $A$. Orient the edges of $T$ away from some root vertex
  $r$. Arbitrarily label the vertices \A{r} and \B{r} of $G$. Let $w$
  be a non-root leaf of $T$. Label each vertex of $T-w$ by
  induction. Now $w$ has one incoming arc $(v,w)$. Some vertex in the
  branch set of $G$ corresponding to $v$ is adjacent to some vertex of
  $G$ in the  branch set corresponding to $w$. Label \A{w} and \B{w} so that there
  is an edge in $G$ between \A{v} and \B{w}, or between \B{v} and
  \A{w}.

  For $v\in V(A)$ and $j\in[\verts{C}]$ (and thus $j\in[\verts{B}]$), 
let $P\langle{v,j}\rangle$ be
  the subgraph of \TCP{G}{H}{I} induced by
$$\{(\A{v},\A{y_j},z):z\in V(I)\}\enspace,$$
and let $Q\langle{v,j}\rangle$ be the subgraph of \TCP{G}{H}{I}
induced by
$$\{(\B{v},y,\A{z_j}):y\in V(H)\}\enspace.$$
For $i\in[k]$ (and thus $i\leq \verts{B}$) and $j\in[\verts{C}]$, let
$R\langle{i,j}\rangle$ be the subgraph of \TCP{G}{H}{I} induced by
$$\{(v,\B{y_i},\B{z_j}):v\in V(G)\}\enspace,$$
and let $X\langle{i,j}\rangle$ be the subgraph of \TCP{G}{H}{I}
induced by
$$\cup\{
P\langle{v,j}\rangle \cup Q\langle{v,j}\rangle \cup
R\langle{i,j}\rangle:v\in f^{-1}(i)\}\enspace,$$ as illustrated in
\figref{ThreeDimGrid}. We now prove that the $X\langle{i,j}\rangle$
are the branch sets of a clique minor in \TCP{G}{H}{I}.

\Figure{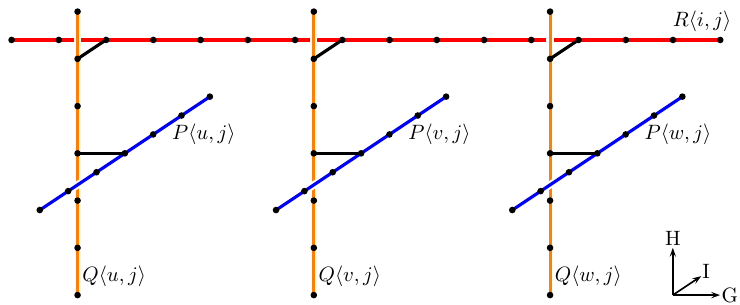}{The branch set $X\langle i,j\rangle$ in
  \propref{ThreeDimGrid}, where $f^{-1}(i)=\{u,v,w\}$.}

First we prove that each $X\langle{i,j}\rangle$ is connected.  Observe
that each $P\langle{v,j}\rangle$ is a copy of $I$ and each
$Q\langle{i,j}\rangle$ is a copy of $H$, and are thus connected.
Moreover, the vertex $(\A{v},\A{y_j},\A{z_j})$ in
$P\langle{v,j}\rangle$ is adjacent to the vertex
$(\B{v},\A{y_j},\A{z_j})$ in $Q\langle{v,j}\rangle$. Thus
$P\langle{v,j}\rangle\cup Q\langle{v,j}\rangle$ is connected.  Now
each $R\langle{i,j}\rangle$ is a copy of $G$, and is thus connected.
For each $v\in f^{-1}(i)$, the vertex $(\B{v},\B{y_i},\A{z_j})$ which
is in $Q\langle{v,j}\rangle$, is adjacent to $(\B{v},\B{y_i},\B{z_j})$
which is in $R\langle{i,j}\rangle$. Thus $X\langle{i,j}\rangle$ is
connected.

Now consider distinct subgraphs $X\langle{i,j}\rangle$ and
$X\langle{i',j'}\rangle$.  We first prove that $X\langle{i,j}\rangle$
and $X\langle{i',j'}\rangle$ are disjoint.  Distinct subgraphs
$P\langle{v,j}\rangle$ and $P\langle{w,j'}\rangle$ are disjoint since
the first two coordinates of every vertex in $P\langle{v,j}\rangle$
are $(\A{v},\A{y_j})$, which are unique to $(v,j)$.  Similarly,
distinct subgraphs $Q\langle{v,j}\rangle$ and $Q\langle{w,j'}\rangle$
are disjoint since the first and third coordinates of every vertex in
$Q\langle{v,j}\rangle$ are $(\B{v},\A{z_j})$, which are unique to
$(v,j)$.  Every $P\langle{v,j}\rangle$ is disjoint from every
$Q\langle{w,j'}\rangle$ since the first coordinate of every vertex in
$P\langle{v,j}\rangle$ is positive, while the first coordinate of
every vertex in $Q\langle{w,j}\rangle$ is negative.  Observe that
$R\langle{i,j}\rangle$ and $R\langle{i',j'}\rangle$ are disjoint since
the second and third coordinates of every vertex in
$R\langle{i,j}\rangle$ are $(\B{y_i},\B{z_j})$, which are unique to
$(i,j)$. Also $R\langle{i,j}\rangle$ is disjoint from every
$P\langle{v,j'}\rangle\cup Q\langle{v,j'}\rangle$ since the second and
third coordinate of every vertex in $R\langle{i,j}\rangle$ are
negative, while every vertex in $P\langle{v,j'}\rangle\cup
Q\langle{v,j'}\rangle$ has a positive second or third
coordinate. Therefore distinct $X\langle{i,j}\rangle$ and
$X\langle{i',j'}\rangle$ are disjoint.

It remains to prove that distinct subgraphs $X\langle{i,j}\rangle$ and
$X\langle{i',j'}\rangle$ are adjacent.  If $i=i'$ then the vertex
$(\A{v},\A{y_j},\A{z_{j'}})$, which is in $P\langle{i,j}\rangle$, is
adjacent to the vertex $(\B{v},\A{y_j},\A{z_{j'}})$, which is in
$Q\langle{i',j'}\rangle$.  Now assume that $i\ne i'$. Then $f(v)=i$
and $f(w)=i'$ for some edge $vw$ of $A$.  By the labelling of vertices
in $G$, without loss of generality, there is an edge in $G$ between
\A{v} and \B{w}.  Thus the vertex $(\A{v},\A{y_j},\A{z_{j'}})$, which
is in $P\langle{v,j}\rangle\subset X\langle{i,j}\rangle$, is adjacent
to the vertex $(\B{w},\A{y_j},\A{z_{j'}})$, which is in
$Q\langle{w,j}\rangle\subset X\langle{i',j'}\rangle$.  In both cases,
$X\langle{i,j}\rangle$ and $X\langle{i',j'}\rangle$ are adjacent.

Hence the $X\langle{i,j}\rangle$ are the branch sets of a clique minor
in \TCP{G}{H}{I}. Thus $\eta(\TCP{G}{H}{I})\geq k\cdot \verts{C}$.
\end{proof}

Now consider the Hadwiger number of the three-dimensional grid graph
\TCP{P_n}{P_n}{P_n}. Prior to this work the best lower and upper
bounds on $\eta(\TCP{P_n}{P_n}{P_n})$ were $\Omega(n)$ and
\Oh{n^{3/2}} respectively \citep{CS-DM07,Raju06,CR05}. The next result
improves this lower bound by a $\Theta(\sqrt{n})$ factor, thus
determining $\eta(\TCP{P_n}{P_n}{P_n})$ to within a factor of $4$
(ignoring lower order terms).

\begin{proposition}
  \proplabel{ThreeDimGrid} For all integers $n\geq m\geq1$,
$$\half n\sqrt{m} -\Oh{n+\sqrt{m}}
< \eta(\TCP{P_n}{P_n}{P_m})\leq 2n\sqrt{m}+3\enspace.$$
\end{proposition}

\begin{proof}
  The upper bound follows from \lemref{UpperBound} since
  \TCP{P_n}{P_n}{P_m} has $n^2m$ vertices and maximum degree $6$.  Now
  we prove the lower bound. $P_m$ has a $P_{\floor{m/2}}$-minor with
  two vertices in each branch set. By \eqnref{PseudoPath},
  $\psi(P_{\floor{m/2}})>\sqrt{m-3}-2$. By \thmref{PseudoAchromatic}
  with $G=P_m$, $A=P_{\floor{m/2}}$, $H=I=P_n$, and
  $B=C=P_{\floor{n/2}}$ (and since
  $\psi(P_{\floor{m/2}})\leq\floor{\tfrac{m}{2}}\leq\floor{\tfrac{n}{2}}=\verts{B}$),
$$\eta(\TCP{P_n}{P_n}{P_m})\geq (\sqrt{m-3}-2)\floor{\tfrac{n}{2}}
=\half n\sqrt{m} -\Oh{n+\sqrt{m}}\enspace,$$ as desired.
\end{proof}

Here is another scenario when tight bounds for three-dimensional grids
can be obtained.

\begin{proposition}
  \proplabel{ThreeDimGridAgain} For all integers $n\geq m\geq1$ such
  that $n\leq\quarter m^2$,
$$\half m\sqrt{n} -\Oh{m+\sqrt{n}}
< \eta(\TCP{P_n}{P_m}{P_m})\leq 2m\sqrt{n}+3\enspace.$$
\end{proposition}

\begin{proof}
  The upper bound follows from \lemref{UpperBound} since
  \TCP{P_n}{P_m}{P_m} has $m^2n$ vertices and maximum degree $6$.  For
  the lower bound, apply \thmref{PseudoAchromatic} with $G=P_n$,
  $A=P_{\floor{n/2}}$, $H=I=P_m$, and $B=C=P_{\floor{m/2}}$. By
  \eqnref{PseudoPath}, $\psi(P_{\floor{n/2}})>\sqrt{n-3}-2$. Thus
$$\eta(\TCP{P_n}{P_n}{P_n})\geq 
\min\{(\sqrt{n-3}-2),\floor{\tfrac{m}{2}}\}\cdot
\floor{\tfrac{m}{2}}\enspace.$$ Since $n\leq\quarter m^2$,
$$\eta(\TCP{P_n}{P_n}{P_n})\geq 
(\sqrt{n-3}-2)\cdot \floor{\tfrac{m}{2}} =\half m\sqrt{n}
-\Oh{m+\sqrt{n}}\enspace,$$ as desired.\end{proof}

Now consider the Hadwiger number of $P_n^d$ for odd $d$. Prior to this
work the best lower and upper bounds on $\eta(P_n^d)$ were
$\Omega(n^{(d-1)/2})$ and \Oh{\sqrt{d}\,n^{d/2}} respectively
\citep{CS-DM07,CR05,Raju06}. The next result improves this lower bound
by a $\Theta(\sqrt{n})$ factor, thus determining $\eta(P_n^d)$ to
within a factor of $2\sqrt{2(d-1)}$ (ignoring lower order terms).
\begin{theorem}
  \thmlabel{OddDimGrid} For every integer $n\geq1$ and odd integer
  $d\geq3$,
$$\half n^{d/2}-\Oh{n^{(d-1)/2}}
< \eta(P_n^d)\leq \sqrt{2(d-1)}\,n^{d/2}+3\enspace.$$
\end{theorem}

\begin{proof}
  The upper bound follows from \lemref{UpperBound} since
  $\verts{P_n^d}=n^d$ and $\Delta(P_n^d)=2d$.  Now we prove the lower
  bound. Let $G:=P_n$ and $A:=P_{\floor{n/2}}$. By
  \eqnref{PseudoPath}, $\psi(A)>\sqrt{n-3}-2$. Let
  $H:=P_n^{(d-1)/2}$. Every grid graph with $m$ vertices has a
  matching of $\floor{\tfrac{m}{2}}$ edges. (\emph{Proof}: induction
  on the number of dimensions.)\ Thus $H$ has a minor $B$ with
  $\floor{\half n^{(d-1)/2}}$ vertices, and two vertices in each
  branch set. By \thmref{PseudoAchromatic} with $I=H$ and $C=B$ (and
  since $\psi(P_{\floor{n/2}})\leq\floor{\tfrac{n}{2}}\leq\floor{\half
    n^{(d-1)/2}}$),
$$\eta(P_n^d)\geq (\sqrt{n-3}-2)\floor{\half n^{(d-1)/2}}
=\half n^{d/2}-\Oh{n^{(d-1)/2}}\enspace,$$ as desired.\end{proof}

\section{Star Minors and Dominating Sets} \seclabel{StarMinors}

Recall that $S_t$ is the star graph with $t$ leaves. Consider the
Hadwiger number of the product of $S_t$ with a general graph.

\begin{lemma}
  \lemlabel{GraphStar} For every connected graph $G$ and for every
  integer $t\geq 1$,
$$\eta(\CP{G}{S_t})\geq \min\{\verts{G},t+1\}\enspace.$$
\end{lemma}

\begin{proof}
  Let $k:= \min\{\verts{G},t+1\}$.  Let $V(G):=\{v_1,v_2,\dots,v_n\}$
  and $V(S_t):=\{r\}\cup[t]$, where $r$ is the root.  For $i\in[k-1]$,
  let $X_i$ be the subgraph of \CP{G}{S_t} induced by
  $\{(v_i,r)\}\cup\{(v_j,i):j\in[n]\}$, Since $G$ is connected, and
  $(v_i,r)$ is adjacent to $(v_i,i)$, each $X_i$ is connected.  Let
  $X_k$ be the subgraph consisting of the vertex $(v_k,r)$.  For
  distinct $i,j\in[k]$ with $i<j$ (and thus $i\in[k-1]\subseteq[t]$),
  the subgraphs $X_i$ and $X_j$ are disjoint, and vertex $(v_j,i)$,
  which is in $X_i$, is adjacent to $(v_j,r)$, which is in $X_j$.
  Thus $X_1,\dots,X_k$ are the branch sets of a $K_k$-minor in
  \CP{G}{S_t}, as illustrated in \figref{PathStar} when $G$ is a path.
\end{proof}

\Figure{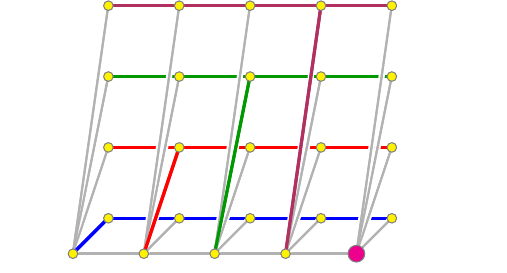}{A $K_5$-minor in \CP{P_5}{S_4}.}

Note that the lower bound in \lemref{GraphStar} is within a constant
factor of the upper bound in \lemref{UpperBound} whenever
$\edges{G}=\Theta(\verts{G})=\Theta(t)$. In particular, if $G$ is a
tree with at least $t+1$ vertices, then
$$t+1\leq\eta(\CP{G}{S_t})\leq\eta(\CP{G}{K_{t+1}})\leq t+2\enspace,$$
where the upper bound is proved in \thmref{TreeComplete} below.  When
$G$ is another star, \citet{Ivanco88} determined the Hadwiger number
precisely.  We include the proof for completeness.

\begin{lemma}[\citep{Ivanco88}]
  \lemlabel{StarStar} For all integers $n \geq m\geq 2$,
$$\eta(\CP{S_n}{S_m})=m+2\enspace.$$
\end{lemma}

\begin{proof}
  Let $V(S_n):=\{r\}\cup[n]$, where $r$ is the root vertex. Observe
  that for all $i\in[n]$ and $j\in[m]$, vertex $(i,j)$ has degree $2$;
  it is adjacent to $(r,j)$ and $(i,r)$. In every graph except $K_3$,
  contracting an edge incident to a degree-$2$ vertex does not change
  the Hadwiger number. Thus replacing the path $(r,j)(i,j)(i,r)$ by
  the edge $(r,j)(i,r)$ does not change the Hadwiger number. Doing so
  gives $K_{1,m,n}$. \citet{Ivanco88} proved that
  $\eta(K_{1,m,n})=m+2$. Thus $\eta(\CP{S_n}{S_m})=m+2$. In fact,
  \citet{Ivanco88} determined the Hadwiger number of every complete
  multipartite graph (a result rediscovered by the author
  \citep{Wood-GC07}).
\end{proof}

For every graph $G$, let \STAR{G} be the maximum integer $t$ for which
$S_t$ is minor of $G$. Since a vertex and its neighbours form a star,
$\STAR{G}\geq\Delta(G)$.  Thus \twolemref{GraphStar}{StarStar} imply:

\begin{corollary}
  \corlabel{GraphStarMinor} For all connected graphs $G$ and $H$,
  \begin{align*}
    \eta(\CP{G}{H})\geq&\min\{\STAR{G}+1,\verts{H}\}\geq\min\{\Delta(G)+1,\verts{H}\},\text{ and }\\
    \eta(\CP{G}{H})\geq& \min\{\STAR{G},\STAR{H}\}+2\geq
    \min\{\Delta(G),\Delta(H)\}+2\enspace.
  \end{align*}
\end{corollary}

As an aside, we now show that star minors are related to radius and
bandwidth.  Let $G$ be a connected graph. The \emph{radius} of $G$,
denoted by \rad{G}, is the minimum, taken over all vertices $v$ of
$G$, of the maximum distance between $v$ and some other vertex of
$G$. Each vertex $v$ that minimises this maximum distance is a
\emph{centre} of $G$.

\begin{lemma}
  \lemlabel{RadiusBandwidth} For every connected graph $G$ with at
  least one edge,
  \begin{align*}
    \textup{(a)}\quad&\verts{G}\leq \STAR{G}\cdot \rad{G}+1\\
    \textup{(b)}\quad&\bw{G}\leq 2\cdot\STAR{G}-1\enspace.
  \end{align*}
\end{lemma}

\begin{proof}
  First we prove (a). Let $v$ be a centre of $G$. For
  $i\in[0,\rad{G}]$, let $V_i$ be the set of vertices at distance $i$
  from $v$. Thus the $V_i$ are a partition of $V(G)$, and
  $|V_0|=1$. For each $i\in[\rad{G}]$, contracting $V_0,\dots,V_{i-1}$
  into a single vertex and deleting $V_{i+1},\dots,V_{\rad{G}}$ gives
  a $S_{|V_i|}$-minor. Thus $\STAR{G}\geq|V_i|$. Hence
  $\verts{G}=\sum_i|V_i|\leq 1+\rad{G}\cdot \STAR{G}$.

  Now we prove (b). Let $(v_1,\dots,v_n)$ be a linear ordering of
  $V(G)$ such that if $v_a\in V_i$ and $v_b\in V_j$ with $i<j$, then
  $a<b$. Consider an edge $v_av_b\in E(G)$. Say $v_a\in V_i$ and
  $v_b\in V_j$. Without loss of generality, $i\leq j$. By
  construction, $j\leq i+1$. Thus $b-a\leq
  |V_i|+|V_{i+1}|-1\leq2\cdot\STAR{G}-1$. Hence
  $\bw{G}\leq2\cdot\STAR{G}-1$.
\end{proof}

Note that \lemref{RadiusBandwidth}(a) is best possible for
$G=P_{2n+1}$, which has $\STAR{G}=2$ and $\rad{G}=n$, or for $G=K_n$,
which has $\STAR{G}=n-1$ and $\rad{G}=1$. \corref{StarNumber} and
\lemref{RadiusBandwidth} imply that the product of graphs with small
radii or large bandwidth has large Hadwiger number; we omit the
details.

Star minors are related to connected dominating sets (first defined by
\citet{SW-JMPS79}). Let $G$ be a graph. A set of vertices $S\subseteq
V(G)$ is \emph{dominating} if each vertex in $V(G)-S$ is adjacent to a
vertex in $S$. The \emph{domination number} of $G$, denoted by
$\gamma(G)$, is the minimum cardinality of a dominating set of $G$. If
$S$ is dominating and $G[S]$ is connected, then $S$ and $G[S]$ are
\emph{connected dominating}. Only connected graphs have connected
dominating sets. The \emph{connected domination number} of a connected
graph $G$, denoted by \cd{G}, is the minimum cardinality of a
connected dominating set of $G$. Finally, let $\ell(G)$ be the maximum
number of leaves in a spanning tree of $G$ (where $K_1$ is considered
to have no leaves, and $K_2$ is considered to have one
leaf). \citet{HL84} proved that $\cd{G}=\verts{G}-\ell(G)$. We extend
this result as follows.

\begin{lemma}
  \lemlabel{StarMinorLeavesSpanningTree} \lemlabel{StarMinorConnDom}
  For every connected graph $G$,
$$\STAR{G}=\verts{G}-\cd{G}=\ell(G)\enspace.$$
\end{lemma}

\begin{proof}
  Let $T$ be a spanning tree of $G$ with $\ell(G)$ leaves.  Let $S$ be
  the set of non-leaf vertices of $T$.  Thus $S$ is a connected
  dominating set of $G$, implying $\cd{G}\leq\verts{G}-\ell(G)$ and
  $\ell(G)\leq\verts{G}-\cd{G}$.

  Say $S$ is a connected dominating set in $G$ of order $\cd{G}$.
  Contracting $G[S]$ gives a $S_{\verts{G}-\cd{G}}$-minor in $G$.
  Thus $\STAR{G}\geq \verts{G}-\cd{G}$.

  Now suppose that $X_0,X_1,\dots,X_{\STAR{G}}$ are the branch sets of
  a $S_{\STAR{G}}$-minor in $G$, where $X_0$ is the root. By
  \lemref{TotalMinor}, we may assume that every vertex of $G$ is in
  some $X_i$. Let $T_i$ be a spanning tree of each $G[X_i]$. Each
  $T_i$ has at least two leaves, unless $\verts{T_i}\leq 2$. Let $T$
  be the tree obtained from $\cup_iT_i$ by adding one edge between
  $T_0$ and $T_i$ for each $i\in[\STAR{G}]$. Each such $T_i$
  contributes at least one leaf to $T$. Thus $T$ has at least
  $\STAR{G}$ leaves, implying $\ell(G)\geq\STAR{G}$.
\end{proof}

\begin{corollary}
  \corlabel{StarMinorTree} For every tree $T\ne K_2$, $\STAR{T}$
  equals the number of leaves in $T$.
\end{corollary}

\corref{GraphStarMinor} and \lemref{StarMinorConnDom} imply:

\begin{corollary}
  \corlabel{StarNumber} For all connected graphs $G$ and $H$,
  \begin{align*}
    \eta(\CP{G}{H})\geq & \min\{\verts{G}-\cd{G}+1,\verts{H}\}\text{ and }\\
    \eta(\CP{G}{H})\geq &
    \min\{\verts{G}-\cd{G},\verts{H}-\cd{H}\}+2\enspace.
  \end{align*}
\end{corollary}

We now show that a connected dominating set in a product can be
constructed from a dominating set in one of its factors.

\begin{lemma}
  \lemlabel{ConDomProd} For all connected graphs $G$ and $H$,
$$\gamma_c(\CP{G}{H})\leq (\verts{G}-1)\cdot\gamma(H)+\verts{H}\enspace.$$
\end{lemma}

\begin{proof}
  Let $S$ be a minimum dominating set in $H$. Let $v$ be an arbitrary
  vertex in $G$. Consider the set of vertices in \CP{G}{H},
$$T:=\{(x,y):x\in V(G)-\{v\},y\in S\}\cup\{(v,y):y\in V(H)\}\enspace.$$
Then $|T|=(\verts{G}-1)\cdot\gamma(H)+\verts{H}$. First we prove that
$T$ is dominating. Consider a vertex $(x,y)$ of \CP{G}{H}. If $y\in S$
then $(x,y)\in T$. Otherwise, $y$ is adjacent to some vertex $z\in S$,
in which case $(x,y)$ is adjacent to $(x,z)$, which is in $T$. Thus
$T$ is dominating. Now we prove that $T$ is connected. Since $G$ is
connected, for each vertex $x$ of $G$, there is a path $P(x,v)$
between $x$ and $v$ in $G$. Consider distinct vertices $(x,y)$ and
$(x',y')$ in \CP{G}{H}. Since $H$ is connected there is a path
$Q(y,y')$ between $y$ and $y'$ in $H$. Thus $$\{(s,y):s\in
P(x,v)\}\cup\{(v,t):t\in Q(y,y')\}\cup \{(s,y'):s\in P(x',v)\}$$ is a
path between $(x,y)$ and $(x',y')$ in \CP{G}{H}, all of whose vertices
are in $T$. Thus $T$ is a connected dominating set.
\end{proof}

\corref{StarNumber} (with $G=\CP{A}{B}$ and $H=C$) and
\lemref{ConDomProd} (with $G=A$ and $H=B$) imply:

\begin{corollary}
  \corlabel{TripleCartProd} For all connected graphs $A,B,C$,
$$\eta(\TCP{A}{B}{C})\geq
\min\{\verts{A}\cdot\verts{B}-(\verts{A}-1)\cdot\gamma(B)-\verts{B}+1,
\verts{C}\}\enspace.$$
\end{corollary}

There are hundreds of theorems about domination in graphs that may be
used in conjunction with the above results to construct clique minors
in products; see the monograph \citep{Domination}. Instead, we now
apply perhaps the most simple known bound on the order of dominating
sets to conclude tight bounds on $\eta(G^d)$ whenever $d\geq4$ is even
and $G$ has bounded average degree.


\begin{theorem}
  Let $G\neq K_1$ be a connected graph with $n$ vertices and average
  degree $\delta$.  Then for every even integer $d\geq4$,
$$\half n^{d/2}-n^{d/2-1}+2\leq \eta(G^d) \leq \sqrt{\delta d-2}\,n^{d/2}+3\enspace.$$
\end{theorem}

\begin{proof}
  The upper bound follows from \lemref{UpperBound} since $G^d$ has
  average degree $\delta d\geq 4$.  Now we prove the lower
  bound. \citet{Ore} observed that for every connected graph $H\neq
  K_1$, the smaller colour class in a $2$-colouring of a spanning tree
  of $H$ is dominating in $H$. Thus $\gamma(H)\leq\half
  \verts{H}$. This observation with $H=G^{d/2-1}$ gives
$$\gamma(G^{d/2-1})\leq\half\verts{G^{d/2-1}}=\half n^{d/2-1}\enspace.$$ 
By \lemref{ConDomProd},
\begin{align*}
  \gamma_c(G^{d/2}) \leq
  (\verts{G}-1)\cdot\gamma(G^{d/2-1})+\verts{G^{d/2-1}} \leq\;&
  (n-1)(\half n^{d/2-1})+n^{d/2-1} \\
  = \;& \half n^{d/2}+\half n^{d/2-1}\enspace.
\end{align*}
By \corref{StarNumber},
\begin{align*}
  \eta(G^d) \geq \verts{G^{d/2}}-\cd{G^{d/2}}+2 \geq\;&
  n^{d/2}-\big(\half n^{d/2}+\half n^{d/2-1}\big) \\
  =\;& \half n^{d/2}-\half n^{d/2-1}+2\enspace,
\end{align*}
as desired.
\end{proof}

\section{Dominating Sets and Clique Minors in Even-Dimensional Grids}
\seclabel{Domination}

The results in \secref{StarMinors} motivate studying dominating sets
in grid graphs. First consider the one-dimensional case of $P_n$. It
is well known and easily proved that
$\gamma(P_n)=\ceil{\frac{n}{3}}$. Thus, by \lemref{ConDomProd}, for
every connected graph $G$,
$$\gamma_c(\CP{G}{P_n})\leq (\verts{G}-1)\ceil{\tfrac{n}{3}}+n\enspace.$$ 
In particular,
$$\gamma_c(\CP{P_m}{P_n})
\leq (m-1)\ceil{\tfrac{n}{3}}+n \leq (m-1)(\tfrac{n+2}{3})+n
=\third(nm+2m+2n-2)\enspace.$$ Hence \corref{StarNumber} with
$G=\CP{P_n}{P_m}$ implies the following bound on the Hadwiger number
of the $4$-dimensional grid:
$$\eta(\CP{P_n}{P_m}\square\CP{P_n}{P_m})
\geq nm-\third(nm+2m+2n-2)+2 =\tfrac{2}{3}(nm-m-n+4)\enspace.$$ This
result improves upon the bound in \thmref{EvenGrid} by a constant
factor.


Dominating sets in 2-dimensional grid graphs are well studied
\citep{CGZ-AC01, VW-DM07, CCEH-AC94, HR-DMGT04, HR-DMGT03, CC-JGT93,
  JK-JGT86, JK-AC84, ES-JEMS02, Shaheen-CN00, ES-JEMS99}. Using the
above technique, these results imply bounds on the Hadwiger number of
the $6$-dimensional grid. We omit the details, and jump to the general
case.

We first construct a dominating set in a general grid graph.

\begin{lemma}
  \lemlabel{DomGenGrid} Fix integers $d\geq1$ and
  $n_1,n_2,\dots,n_d\geq1$.  Let $S$ be the set of vertices
$$S:=\{(x_1,x_2,\dots,x_d):x_i\in[n_i],i\in[d],
\sum_{i\in[d]}i\cdot x_i\equiv 0\pmod{2d+1}\}\enspace.$$ For
$j\in[d]$, let $B_j$ be the set of vertices
\begin{align*}
  B_j:=\big\{(x_1,\dots,x_{j-1},1,x_{j+1},\dots,x_d)\;:\;&
  x_i\in[n_i],i\in[d]-\{j\},\\
  & \sum_{i\in[d]-\{j\}}\!\!\!i\cdot x_i\equiv
  0\pmod{2d+1}\big\}\enspace,
\end{align*}
and let $C_j$ be the set of vertices
\begin{align*}
  C_j:=\big\{(x_1,\dots,x_{j-1},n_j,x_{j+1},&\dots,x_d)\;:\; x_i\in[n_i],i\in[d]-\{j\},\\
  & \sum_{i\in[d]-\{j\}}\!\!\!i\cdot x_i\equiv
  -j(n_j+1)\pmod{2d+1}\big\}\enspace,
\end{align*}
Let $T:=\cup_j(S\cup B_j\cup C_j)$. Then $T$ is dominating in
\CCP{P_{n_1}}{P_{n_2}}{P_{n_d}}.
\end{lemma}

\begin{proof}
  Consider a vertex $x=(x_1,x_2,\dots,x_d)$ not in $S$.  We now prove
  that $x$ has neighbour in $S$, or $x$ is in some $B_j\cup C_j$.  Now
  $x_i\in[n_i]$ for each $i\in[d]$, and for some $r\in[2d]$,
$$\sum_{i=1}^di\cdot x_i\equiv r\pmod{2d+1}\enspace.$$

First suppose that $r\in[d]$. Let $j:=r$. Thus
$$j\cdot (x_j-1)+\sum_{i\in[d]-\{j\}}i\cdot x_i\equiv 0\pmod{2d+1}\enspace.$$
Hence, if $x_j\ne 1$ then
$(x_1,\dots,x_{j-1},x_j-1,x_{j+1},\dots,x_d)$ is a neighbour of $x$ in
$S$, and $x$ is dominated.  If $x_j=1$ then $x$ is in $B_j\subset T$.

Now assume that $r\in[d+1,2d]$. Let $j:=2d+1-r\in[d]$.  Thus $r\equiv
-j\pmod{2d+1}$, and
$$j\cdot (x_j+1)+\sum_{i\in[d]-\{j\}}i\cdot x_i\equiv 0\pmod{2d+1}\enspace.$$
Hence, if $x_j\ne n_j$ then
$(x_1,\dots,x_{j-1},x_j+1,x_{j+1},\dots,x_d)$ is a neighbour of $x$ in
$S$, and $x$ is dominated.  If $x_j=n_j$ then $x$ is in $C_j\subset
T$.

Thus every vertex not in $T$ has a neighbour in $S\subset T$, and $T$
is dominating.
\end{proof}

We now determine the size of the dominating set in
\lemref{DomGenGrid}.

\begin{lemma}
  \lemlabel{EasyCounting} For integers $r\geq2$, $d\geq 1$, $c$, and
  $n_1,\dots,n_d\geq 1$, define
  \begin{align*}
    Q(n_1,\dots,n_d; c; r):=\{(x_1,x_2,\dots,x_d):&x_i\in[n_i],i\in[d],\\
    &\sum_{i\in[d]}i\cdot x_i\equiv c\pmod{r}\}\enspace.
  \end{align*}
  If each $n_i\equiv 0\pmod{r}$ then for every integer $c$,
$$|Q(n_1,\dots,n_d; c; r)|=\frac{1}{r}\prod_{i\in[d]}n_i\enspace.$$
\end{lemma}

\begin{proof}
  We proceed by induction on $d$. First suppose that $d=1$. Without
  loss of generality, $c\in[r]$. Then
  \begin{align*}
    Q(n_1;c;r) =\{x\in[n_1]:x\equiv c\pmod{r}\} =\{r\cdot
    y+c:y\in[0,\tfrac{n_1}{r}-1]\}\enspace.
  \end{align*}
  Thus $|Q(n_1;c;r)|=\tfrac{n_1}{r}$, as desired. Now assume that
  $d\geq2$. Thus
  \begin{align*}
    &|Q(n_1,\dots,n_d;c;r)|\\
    =\,& |\{(x_1,x_2,\dots,x_d):x_i\in[n_i],i\in[d],
    \sum_{i\in[d]}i\cdot x_i\equiv c\pmod{r}\}|\\
    =\,&
    \sum_{x_d\in[n_d]}|\{(x_1,x_2,\dots,x_d):x_i\in[n_i],i\in[d-1],\\
    &\hspace*{40mm}    \sum_{i\in[d-1]}i\cdot x_i\equiv (c-d\cdot x_d)\pmod{r}\}|\\
    =\,& \sum_{x_d\in[n_d]}|Q(n_1,\dots,n_{d-1};c-d\cdot
    x_d;r)|\enspace.
  \end{align*}
  By induction,
  \begin{align*}
    |Q(n_1,\dots,n_d;c;r)| =
    \sum_{x_d\in[n_d]}\frac{1}{r}\prod_{i\in[d-1]}n_i =
    \frac{1}{r}\prod_{i\in[d]}n_i\enspace,
  \end{align*}
  as desired.
\end{proof}

\begin{lemma}
  \lemlabel{EasyDomGenGrid} Let $G:=\CCP{P_{n_1}}{P_{n_2}}{P_{n_d}}$
  for some integers $d\geq1$ and $n_1,n_2,\dots,n_d\geq1$, where each
  $n_i\equiv0\pmod{2d+1}$. Then
$$\gamma(G)\leq
\frac{\verts{G}}{2d+1}\bigg(1+\sum_{j\in[d]}\frac{2}{n_j}\bigg)\enspace.$$
\end{lemma}

\begin{proof}
  Using the notation in \lemref{DomGenGrid}, by \lemref{EasyCounting}
  applied three times,
$$|S|=\frac{1}{2d+1}\prod_{i\in[d]}n_i=\frac{\verts{G}}{2d+1}\enspace,$$ 
and for each $j\in[d]$,
$$|B_j|,|C_j|
= \frac{1}{2d+1}\prod_{i\in[d]-\{j\}}\!\!\!n_i =
\frac{\verts{G}}{(2d+1)n_j}\enspace.$$ Thus
$$|T|\leq|S|+\sum_{j\in[d]}|B_j|+|C_j|=
\frac{\verts{G}}{2d+1}\bigg(1+\sum_{j\in[d]}\frac{2}{n_j}\bigg)\enspace,$$
as desired.\end{proof}

If $n_1,\dots,n_d$ are large compared to $d$, then
\lemref{EasyDomGenGrid} says that $\gamma(G)\leq\frac{\verts{G}}{2d+1}
+ o(\verts{G})$.  This bound is best possible since $\gamma(H)\geq
\frac{\verts{H}}{\Delta(H)+1}$ for every graph $H$ (and $G$ has
maximum degree $2d$).


From the dominating set given in \lemref{EasyDomGenGrid} we construct
a connected dominating set as follows.

\begin{lemma}
  \lemlabel{EasyConDomGenGrid} Let
  $G:=\CCP{P_{n_1}}{P_{n_2}}{P_{n_d}}$ for some integers $d\geq1$ and
  $n_1\geq n_2\geq\dots\geq n_d\geq1$, where each
  $n_i\equiv0\pmod{2d+1}$. Then
$$\cd{G}<\frac{\verts{G}}{2d-1}\bigg(1+\frac{2d-2}{n_d}+\sum_{j\in[d-1]}\frac{2}{n_j}\bigg)\enspace.$$
\end{lemma}

\begin{proof}
  Let $G':=\CCP{P_{n_1}}{P_{n_2}}{P_{n_{d-1}}}$.  By
  \lemref{EasyDomGenGrid} applied to $G'$,
$$\gamma(G')
\leq\frac{\verts{G'}}{2d-1}\bigg(1+\sum_{j\in[d-1]}\frac{2}{n_j}\bigg)\enspace.$$
By \lemref{ConDomProd} with $H=G'$,
$$\cd{G}\leq (n_d-1)\cdot\gamma(G')+\verts{G'}\enspace.$$
Thus,
\begin{align*}
  \cd{G} &\leq(n_d-1)
  \frac{\verts{G'}}{2d-1}\bigg(1+\sum_{j\in[d-1]}\frac{2}{n_j}\bigg)
  +\verts{G'}\\
  & =\frac{\verts{G}}{2d-1}\bigg(1+\sum_{j\in[d-1]}\frac{2}{n_j}\bigg)
  -\frac{\verts{G'}}{2d-1}\bigg(1+\sum_{j\in[d-1]}\frac{2}{n_j}\bigg)
  +\verts{G'}\\
  & <\frac{\verts{G}}{2d-1}\bigg(1+\sum_{j\in[d-1]}\frac{2}{n_j}\bigg)
  -\frac{\verts{G'}}{2d-1}+\verts{G'}\\
  & =\frac{\verts{G}}{2d-1}\bigg(1+\sum_{j\in[d-1]}\frac{2}{n_j}\bigg)
  +\frac{\verts{G'}\,(2d-2)}{2d-1}\\
  &
  =\frac{\verts{G}}{2d-1}\bigg(1+\frac{2d-2}{n_d}+\sum_{j\in[d-1]}\frac{2}{n_j}\bigg)\enspace,
\end{align*}
as desired.
\end{proof}

Note that \citet{Gravier-DAM02} proved an analogous result to
\lemref{EasyConDomGenGrid} for the total domination number of
multi-dimensional grids. \lemref{EasyConDomGenGrid} leads to the
following bounds on the Hadwiger number of even-dimensional
grids. These lower and upper bounds are within a multiplicative factor
of approximately $2\sqrt{d}$, ignoring lower order terms.

\begin{theorem}
  \thmlabel{EasyGenGrid} Let
  $G:=\CCP{P^2_{n_1}}{P^2_{n_2}}{P^2_{n_d}}$ for some integers
  $d\geq1$ and $n_1\geq n_2\geq\dots\geq n_d\geq1$, where each
  $n_i\equiv0\pmod{2d+1}$. Then
$$\sqrt{\verts{G}}\bigg(1-\frac{1}{2d-1}\bigg)\bigg(1-\frac{1}{n_d}
-\frac{1}{d-1}\!\!\sum_{j\in[d-1]}\!\!\frac{1}{n_j}\bigg)+2 \leq
\eta(G) \leq \sqrt{(4d-2)\, \verts{G}}+3\enspace.$$
\end{theorem}

\begin{proof}
  The upper bound follows from \lemref{UpperBound} since
  $\Delta(G)=4d$.  For the lower bound, let
  $G':=\CCP{P_{n_1}}{P_{n_2}}{P_{n_d}}$.  By
  \lemref{EasyConDomGenGrid} applied to $G'$,
$$\cd{G'}<
\frac{\verts{G'}}{2d-1}\bigg(1+\frac{2d-2}{n_d}+\sum_{j\in[d-1]}\frac{2}{n_j}\bigg).
$$
By \corref{StarNumber} applied to $G'$,
$$\eta(G)\geq  \verts{G'}-\cd{G'}+2\enspace.$$
Thus
\begin{align*}
  \eta(G) &\geq \verts{G'}-
  \frac{\verts{G'}}{2d-1}\bigg(1+\frac{2d-2}{n_d}+\sum_{j\in[d-1]}\frac{2}{n_j}\bigg)+2\\
  &=
  \verts{G'}\bigg(1-\frac{1}{2d-1}-\frac{2d-2}{(2d-1)\,n_d}-\frac{1}{2d-1}\sum_{j\in[d-1]}\frac{2}{n_j}\bigg)+2\\
  &= \sqrt{\verts{G}}\bigg(\frac{2d-2}{2d-1}-\frac{2d-2}{(2d-1)\,n_d}
  -\frac{2}{2d-1}\sum_{j\in[d-1]}\frac{1}{n_j}\bigg)+2\\
  &= \sqrt{\verts{G}}\bigg(1-\frac{1}{2d-1}\bigg)\bigg(1-\frac{1}{n_d}
  -\frac{1}{d-1}\sum_{j\in[d-1]}\frac{1}{n_j}\bigg)+2\enspace,
\end{align*}
as desired.
\end{proof}

\begin{corollary}
  \corlabel{SquareGrid} For all even integers $d\geq4$ and $n\geq1$
  such that $n\equiv0\pmod{2d+1}$,
$$n^{d/2}\Big(1-\frac{1}{d-1}\Big)\Big(1-\frac{2}{n}\Big)+2
\leq \eta(P_n^d) \leq \sqrt{2(d-1)}\,n^{d/2}+3\enspace.$$
\end{corollary}

\section{Hadwiger Number of Products of Complete Graphs}
\seclabel{Hamming}

In this section we consider the Hadwiger number of the product of
complete graphs. First consider the case of two complete graphs.  \CR\
proved that $\eta(\CP{K_n}{K_m})=\Theta(n\sqrt{m})$ for $n\geq m$.  In
particular,
$$\tfrac{1}{4}(n-\sqrt{m})(\sqrt{m}-2) \leq \eta(\CP{K_n}{K_m})\leq 2n\sqrt{m}\enspace.$$
Since $P_{\floor{\sqrt{m}}} \square P_{\floor{\sqrt{m}}} \square K_n$
is a subgraph of \CP{K_n}{K_m}, \lemref{PathPathComplete} below
immediately improves this lower bound to
$$\eta(\CP{K_n}{K_m})\geq
\floor{\tfrac{n}{2}}\floor{\sqrt{m}}\enspace.$$ It is interesting that
(up to a constant factor) \CP{K_n}{K_m} and $P_{\floor{\sqrt{m}}}
\square P_{\floor{\sqrt{m}}} \square K_n$ have the same Hadwiger
number. \citet{CKR-GC08} improved both the lower and upper bound on
$\eta(\CP{K_n}{K_m})$ to conclude the following elegant result.

\begin{theorem}[\citep{CKR-GC08}]
  \thmlabel{CompleteComplete} For all integers $n\geq m\geq1$,
$$\eta(\CP{K_n}{K_m})=(1-o(1))\,n\sqrt{m}\enspace.$$ 
\end{theorem}

\thmref{CompleteComplete} is improved for small values of $m$ in the
following three propositions.

\begin{proposition}
  \proplabel{CompleteEdge} For every integer $n\geq1$,
$$\eta(\CP{K_n}{K_2})=n+1\enspace.$$
\end{proposition}

\begin{proof}
  Say $V(K_n)=[n]$ and $V(K_2)=\{v,w\}$.

  First we prove the lower bound $\eta(\CP{K_n}{K_2})\geq n+1$.  For
  $i\in[n]$, let $X_i$ be the subgraph of \CP{K_n}{K_2} induced by the
  vertex $(i,v)$. Let $X_{n+1}$ be the subgraph of \CP{K_n}{K_2}
  induced by the vertices $(1,w),\dots,(n,w)$. Then
  $X_1,\dots,X_{n+1}$ are branch sets of a $K_{n+1}$-minor in
  \CP{K_n}{K_2}. Thus $\eta(\CP{K_n}{K_2})\geq n+1$.

  It remains to prove the upper bound $\eta(\CP{K_n}{K_2})\leq
  n+1$. Let $X_1,\dots,X_k$ be the branch sets of a complete minor in
  \CP{K_n}{K_2}, where $k=\eta(\CP{K_n}{K_2})$. If every $X_i$ has at
  least two vertices then $k\leq n$ since \CP{K_n}{K_2} has $2n$
  vertices. Otherwise some $X_i$ has only one vertex, which has degree
  $n$ in \CP{K_n}{K_2}. Thus $k\leq n+1$, as desired.
\end{proof}

\begin{proposition}
  \proplabel{TriangleComplete} For every integer $n\geq1$,
$$\eta(K_n\square K_3)=n+2\enspace.$$
\end{proposition}

\begin{proof}
  A $K_{n+2}$-minor in $K_n\square K_3$ is obtained by contracting the
  first row and contracting the second row. Thus $\eta(K_n\square
  K_3)\geq n+2$.

  It remains to prove the upper bound $\eta(K_n\square K_3)\leq n+2.$
  We proceed by induction on $n$. The base case $n=1$ is trivial. Let
  $X_1,\dots,X_k$ be the branch sets of a $K_k$-minor, where
  $k=\eta(K_n\square K_3)$. Without loss of generality, each $X_i$ is
  an induced subgraph.

  Suppose that some column $C$ intersects at most one branch set
  $X_i$. Deleting $C$ and $X_i$ gives a $K_{k-1}$-minor in
  \CP{K_{n-1}}{K_3}. By induction, $k-1\geq n+1$. Thus $k\geq n+2$, as
  desired. Now assume that every column intersects at least two branch
  sets.

  If some branch set has only one vertex $v$, then
  $k\leq1+\deg(v)=n+2$, as desired. Now assume that every branch set
  has at least two vertices.

  Suppose that some branch set $X_i$ has vertices in distinct
  rows. Since $X_i$ is connected, $X_i$ has at least two vertices in
  some column $C$. Now $C$ intersects at least two branch sets, $X_i$
  and $X_j$. Thus $X_j$ intersects $C$ in exactly one vertex
  $v$. Consider the subgraph $X_j-v$. It has at least one
  vertex. Every neighbour of $v$ that is in $X_j$ is in the same row
  as $v$. Since $X_j$ is an induced subgraph, the neighbourhood of $v$
  in $X_j$ is a non-empty clique. Thus $v$ is not a cut-vertex in
  $X_j$, and $X_j-v$ is a non-empty connected subgraph. Hence deleting
  $C$ and $X_i$ gives a $K_{k-1}$-minor in \CP{K_{n-1}}{K_3}. By
  induction, $k-1\geq n+1$. Thus $k\geq n+2$, as desired. Now assume
  that each branch set is contained in some row.

  If every branch set has at least three vertices, then
  $k\leq\frac{1}{3}|V(K_n\square K_3)|=n$, as desired. Now assume that
  some branch set $X_i$ has exactly two vertices $v$ and $w$. Now $v$
  and $w$ are in the same row. There are at most $\frac{n-2}{2}$ other
  branch sets in the same row, since every branch set has at least two
  vertices and is contained in some row. Moreover, $N(v)\cup N(w)$
  contains only four vertices that are not in the same row as $v$ and
  $w$. Thus $k-1\leq\frac{n-2}{2}+4$. That is, $k\leq\frac{n+8}{2}$,
  which is at most $n+2$ whenever $n\geq4$. Now assume that
  $n\leq3$. Since every branch set has at least two vertices,
  $k\leq\half|V(K_n\square K_3)|=\frac{3n}{2}$, which is at most $n+2$
  for $n\leq4$. \end{proof}

\begin{proposition}
  \proplabel{CompleteFour} For every integer $n\geq1$,
$$\floor{\tfrac{3}{2}n}\leq \eta(K_n \square K_4)\leq \tfrac{3}{2}n+7\enspace.$$
\end{proposition}

\begin{proof}
  First we prove the lower bound. Let $p:=\floor{\tfrac{n}{2}}$.  Each
  vertex is described by a pair $(i,x)$ where $i\in[2p]$ and
  $x\in\{a,b,c,d\}$. Distinct vertices $(i,x)$ and $(j,y)$ are
  adjacent if and only if $i=j$ or $x=y$.  As illustrated in
  \figref{Kn4}, for $i\in[p]$, let $X_i$ be the path
  $(2i-1,a)(2i-1,b)(2i,b)(2i,c)$, let $Y_i$ be the edge
  $(2i-1,c)(2i-1,d)$, and let $Z_i$ be the edge $(2i,a)(2i,d)$.  Thus
  each $X_i$, $Y_i$, and $Z_i$ is connected, and each pair of distinct
  subgraphs are disjoint.  Moreover, the vertex $(2i-1,a)$ in $X_i$ is
  adjacent to the vertex $(2j-1,a)$ in $X_j$.  The vertex $(2i,c)$ in
  $X_i$ is adjacent to the vertex $(2j-1,c)$ in $Y_j$.  The vertex
  $(2i-1,a)$ in $X_i$ is adjacent to the vertex $(2j,a)$ in $Z_j$.
  The vertex $(2i-1,c)$ in $Y_i$ is adjacent to the vertex $(2j-1,c)$
  in $Y_j$.  The vertex $(2i-1,d)$ in $Y_i$ is adjacent to the vertex
  $(2j,d)$ in $Z_j$.  And the vertex $(2i,a)$ in $Z_i$ is adjacent to
  the vertex $(2j,a)$ in $Z_j$.  Hence $\{X_i,Y_i,Z_i:i\in[p]\}$ are
  the branch sets of a $K_{3p}$-minor. Therefore
  $\eta(\CP{K_{2p}}{K_4})\geq 3p$. In the case that $n$ is odd, one
  column is unused, implying $\eta(\CP{K_{2p}}{K_4})\geq 3p+1$. It
  follows that $\eta(\CP{K_n}{K_4})\geq \floor{\tfrac{3}{2}n}$ for all
  $n$.

  \Figure{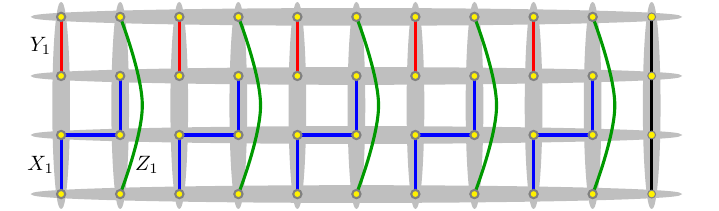}{A $K_{16}$-minor in $\CP{K_{11}}{K_4}$.}

  Now we prove the upper bound. (We make no effort to improve the
  constant $7$.)\ Suppose on the contrary that $\eta(K_n \square
  K_4)>\frac{3}{2}n+7$ for some minimum $n$. Thus $\eta(K_{n'} \square
  K_4)\leq\frac{3}{2}n'+7$ for all $n'<n$.  Consider the branch sets
  of a $K_p$-minor in $K_n \square K_4$, where $p>\frac{3}{2}n+7$.  By
  \lemref{TotalMinor}, we may assume that every vertex is in some
  branch set.

  Suppose that some branch set consists of at most three vertices all
  in a single row. These vertices have at most $n+6$ neighbours in
  total, implying $p-1\leq n+6$, which is a contradiction. Now assume
  that no branch set consists of at most three vertices all in a
  single row. In particular, no branch set is a singleton.

  Suppose that some column $C$ contains exactly three vertices in a
  single branch set $X$. Let $y$ be the vertex in $C\setminus X$.  Let
  $Y$ be the branch set that contains $y$.  The neighbourhood of $y$
  in $Y$ is contained in a single row $R$, and is thus a clique. Hence
  $Y-y$ is connected and non-empty. Since $y$ is the only vertex in
  $Y\cap C$, some neighbour $y'$ of $y$ in $Y$ is in $R$.  If, for
  some branch set $Z$ that does not intersect $C$, some $YZ$-edge is
  incident to $y$, then $Z$ intersects $R$, implying there is an edge
  from $y'$ to $Z$.  Therefore deleting $C$ gives a $K_{p-1}$-minor
  (including the branch set $Y-y$).  Hence
 $$\tfrac{3}{2}(n-1)+7\geq \eta(\CP{K_{n-1}}{K_4})\geq p-1>\tfrac{3}{2}n+6\enspace,$$
 which is a contradiction. Now assume that no column contains exactly
 three vertices in a single branch set.

 Say there are $q$ branch sets, each with exactly two or three
 vertices.  Thus $$4n\geq 2q+4(p-q)>-2q+4(\tfrac{3}{2}n+7)\enspace,$$
 implying $q>n+14$.  Each branch set $X$ with exactly two or three
 vertices contains exactly two vertices in some column $C$ (since no
 branch set consists of at most three vertices all in a single row,
 and no column contains exactly three vertices in a single branch
 set).  We say that $C$ \emph{belongs} to $X$. Since $q\geq n+10$,
 there are distinct columns $C_1,\dots,C_{10}$ that each belong to at
 least two branch sets.

 Say $C_i$ is type-1 if the vertices in the first and second rows (of
 $C_i$) are in the same branch set (which implies that the vertices in
 the third and fourth rows are in the same branch set).  Say $C_i$ is
 type-2 if the vertices in the first and third rows are in the same
 branch set (which implies that the vertices in the second and fourth
 rows are in the same branch set).  Say $C_i$ is type-3 if the
 vertices in the first and fourth rows are in the same branch set
 (which implies that the vertices in the second and third rows are in
 the same branch set).

 At least four of $C_1,\dots,C_{10}$ have the same type.  Without loss
 of generality, $C_1,C_2,C_3,C_4$ are all type-1.  Let $X$ be the
 branch set that contains the vertices in the first and second rows of
 $C_1$. In the case that $|X|=3$, let $D$ be the column that contains
 the vertex in $X\setminus C_1$. Note that $D\neq C_i$ for all
 $i\in\{2,3,4\}$ (since $C_1$ and $C_i$ have the same type and
 $|X|\leq 3$).  For $i\in\{2,3,4\}$, let $Y_i$ be the branch set that
 contains the vertices in the third and fourth rows of $C_i$. Note
 that $Y_2$, $Y_3$ and $Y_4$ are distinct (since exactly one column
 belongs to each $Y_i$). Since $|X|\leq 3$, each vertex in $X$ is in
 the first or second row. For each $i\in\{2,3,4\}$, since $C_i$
 belongs to two branch sets and $|X|\leq 3$, we have $X\cap
 C_i=\emptyset$.  Similarly, each vertex in $Y_i$ is in the third or
 fourth row, and $Y_i\cap C_1=\emptyset$.  Since there is an edge
 between $X$ and $Y_i$, it must be that $|X|=3$, and each $Y_i$
 contains a vertex in the third or fourth row of $D$. That is, two
 vertices are contained in three branch sets.  This contradiction
 completes the proof.
\end{proof}

Now we consider the Hadwiger number of the product of $d$ complete
graphs. Here our lower and upper bounds are within a factor of
$2\sqrt{d}$ (ignoring lower order terms).

\begin{theorem}
  \thmlabel{CompleteProduct} For all integers $d\geq 2$ and $n_1\geq
  n_2\geq\dots\geq n_d\geq2$,
$$\FLOOR{\frac{n_1}{2}}\prod_{i\in[2,d]}\FLOOR{\sqrt{n_i}}
\;\leq\; \eta(\CCP{K_{n_1}}{K_{n_2}}{K_{n_d}}) \;<\; \Big(\sqrt{d}\,
n_1\, \!\!\prod_{i\in[2,d]}\sqrt{n_i}\Big)+3\enspace.$$
\end{theorem}

\begin{proof}
  Let $G:=\CCP{K_{n_1}}{K_{n_2}}{K_{n_d}}$.  Since
  $\verts{G}=\prod_in_i$ and $\Delta(G)=\sum_i(n_i-1)$,
  \lemref{UpperBound} implies the upper bound,
  \begin{align*}
    \eta(G) \;<&\;
    \sqrt{\Big(\sum_{i\in[d]}(n_i-1)\Big)\Big(\prod_{i\in[d]}n_i}\Big)+3\\
    \;<&\; \Big(\sqrt{d n_1}\,\prod_{i\in[d]}\sqrt{n_i}\Big)+3 \\
    \;=&\; \Big(\sqrt{d}\,
    n_1\,\prod_{i\in[2,d]}\sqrt{n_i}\Big)+3\enspace.
  \end{align*}

  For the lower bound, let $p:=\floor{\tfrac{n_1}{2}}$ and
  $k_i:=\floor{\sqrt{n_i}}$ for each $i\in[2,d]$. Let $m_1:=2p$ and
  $m_i:=k_i^2$ for each $i\in[2,d]$.  Observe that each $n_i\geq m_i$.
  Thus it suffices to construct the desired minor in
  \CCP{K_{m_1}}{K_{m_2}}{K_{m_d}}.  Let $V(K_{m_1})=[2p]$, and for
  each $i\in[2,d]$,
  let $$V(K_{m_i})=\{(a_i,b_i):a_i,b_i\in[k_i]\}\enspace.$$ Each
  vertex is described by a vector $(r,a_2,b_2,\dots,a_d,b_d)$ where
  $r\in[m_1]$ and $a_i,b_i\in[k_i]$.  Distinct vertices
  $(r,a_2,b_2,\dots,a_d,b_d)$ and $(s,x_2,y_2,\dots,x_d,y_d)$ are
  adjacent if and only if:

  \begin{enumerate}[(1)]
  \item $a_i=x_i$ and $b_i=y_i$ for each $i\in[2,d]$, or
  \item $r=s$, and for some $i\in[2,d]$, for every $j\neq i$, we have
    $a_j=x_j$ and $b_j=y_j$.
  \end{enumerate}

  \noindent In case (1) the edge is in dimension $1$, and in case (2)
  the edge is in dimension $i$.

  For all $r\in[p]$, $i\in[2,d]$, and $j_i\in[k_i]$, let $A\langle
  r,j_2,\dots,j_d\rangle$ be the subgraph induced by
$$\{(2r,a_2,j_2,a_3,j_3,\dots,a_d,j_d):a_i\in[k_i],i\in[2,d]\}\enspace,$$
let $B\langle r,j_2,\dots,j_d\rangle$ be the subgraph induced by
$$\{(2r-1,j_2,b_2,j_3,b_3,\dots,j_d,b_d):b_i\in[k_i],i\in[2,d]\}\enspace,$$
and let $X\langle r,j_2,\dots,j_d\rangle$ be the subgraph induced by
the vertex set of $A\langle r,j_2,\dots,j_d\rangle\cup B\langle
r,j_2,\dots,j_d\rangle$.

Observe that any two vertices in $A\langle r,j_2,\dots,j_d\rangle$ are
connected by a path of at most $d-1$ edges (in dimensions
$2,\dots,d$).  Thus $A\langle r,j_2,\dots,j_d\rangle$ is connected.
Similarly, $B\langle r,j_2,\dots,j_d\rangle$ is connected.  Moreover,
the dimension-$1$ edge
$$(2r,j_2,j_2,j_3,j_3,\dots,j_d,j_d)(2r-1,j_2,j_2,j_3,j_3,\dots,j_d,j_d)$$ connects 
$A\langle r,j_2,\dots,j_d\rangle$ and $B\langle
r,j_2,\dots,j_d\rangle$.  Hence $X\langle r,j_2,\dots,j_d\rangle$ is
connected.

Consider a pair of distinct subgraphs $X\langle
r,j_2,\dots,j_d\rangle$ and $X\langle s,\ell_2,\dots,\ell_d\rangle$.
By construction they are disjoint.  Moreover, the dimension-$1$ edge
$$(2r,\ell_2,j_2,\ell_3,j_3,\dots,\ell_d,j_d)(2s-1,\ell_2,j_2,\ell_3,j_3,\dots,\ell_d,j_d)$$
connects $A\langle r,j_2,\dots,j_d\rangle$ and $B\langle
s,\ell_2,\dots,\ell_d\rangle$.  Hence the $X\langle
r,j_2,\dots,j_d\rangle$ are branch sets of a clique minor of order
$p\prod_{i=2}^dk_i$. Therefore
\begin{align*}
  \eta(G) \geq p\prod_{i=2}^dk_i =
  \FLOOR{\frac{n_1}{2}}\prod_{i=2}^d\floor{\sqrt{n_i}}\enspace,
\end{align*}
as desired.
\end{proof}


The $d$-dimensional \emph{Hamming} graph is the
product $$H_n^d:=\underbrace{K_n\square K_n\square\dots\square
  K_n}_d\enspace.$$ \citet{CS-DM07} proved the following bounds on the
Hadwiger number of $H_n^d$:
$$n^{\floor{(d-1)/2}}\;\leq\;\eta(H_n^d)\;\leq\;1+\sqrt{d}\,n^{(d+1)/2}\enspace.$$
\thmref{CompleteProduct} improves this lower bound by a $\Theta(n)$
factor; thus determining $\eta(H_n^d)$ to within a $2\sqrt{d}$ factor
(ignoring lower order terms):
$$\half n^{(d+1)/2} - \Oh{n^{d/2}}
\;\leq\; \eta(H_n^d) \;<\;1+\sqrt{d}\,n^{(d+1)/2}\enspace.$$

\section{Hypercubes and Lexicographic Products}

The \emph{$d$-dimensional hypercube} is the graph
$$Q_d:=\underbrace{K_2\square K_2\square\cdots\square K_2}_d\enspace.$$
Hypercubes are both grid graphs and Hamming graphs.  The Hadwiger
number of $Q_d$ was first studied by \citet{CS-GRACO05}.  The best
bounds on $\eta(Q_d)$ are due to \citet{Kotlov-EuJC01}, who proved
that
\begin{equation}
  \eqnlabel{HypercubeLower}
  \eta(Q_d)\,\geq\,
  \begin{cases}
    2^{(d+1)/2} &, d \text{ odd} \\
    3\cdot 2^{(d-2)/2} &, d \text{ even}
  \end{cases}
\end{equation}
and
\begin{equation*}
  \eta(Q_d)\leq 
  \tfrac{5}{2}+\sqrt{2^d(d-3)+\tfrac{25}{4}}\enspace.
\end{equation*}
\citet{Kotlov-EuJC01} actually proved the following more general
result which readily implies \eqnref{HypercubeLower} by induction:

\begin{proposition}[\citep{Kotlov-EuJC01}]
  \proplabel{Kotlov} For every bipartite graph $G$, the strong product
  \SP{G}{K_2} is a minor of \TCP{G}{K_2}{K_2}.
\end{proposition}

\propref{Kotlov} is generalised by the following result with $H=K_2$
(since $G\cdot K_2\cong\SP{G}{K_2}$). See \citep{WYY10} for a
different generalisation.

\begin{proposition}
  \proplabel{LexMinor} For every bipartite graph $G$ and every
  connected graph $H$, the lexicographic product $G\cdot H$ is a minor
  of \TCP{G}{H}{H}.
\end{proposition}

\begin{proof}
  Properly colour the vertices of $G$ \emph{black} and
  \emph{white}. For all vertices $(v,p)$ of $G\cdot H$, let $X\langle
  v,p\rangle$ be the subgraph of \TCP{G}{H}{H} induced by the set
$$\begin{cases}
  \{(v,p,q):q\in V(H)\}	& \text{, if $v$ is black}\\
  \{(v,q,p):q\in V(H)\} & \text{, if $v$ is white.}
\end{cases}$$ We claim that the $X\langle v,p\rangle$ form the branch
sets of a $G\cdot H$-minor in \TCP{G}{H}{H}. First observe that each
$X\langle v,p\rangle$ is isomorphic to $H$, and is thus connected.

Consider distinct vertices $(v,p)$ and $(v',p')$ of $G\cdot H$, where
$v$ is black.

Suppose on the contrary that some vertex $(w,a,b)$ of \TCP{G}{H}{H} is
in both $X\langle v,p\rangle$ and $X\langle v',p'\rangle$. Since
$(w,a,b)$ is in $X\langle v,p\rangle$, we have $a=p$. By construction,
$w=v=v'$. Thus $v'$ is also black, and since $(w,a,b)$ is in $X\langle
v',p'\rangle$, we have $a=p'$. Hence $p=p'$, which contradicts that
$(v,p)$ and $(v',p')$ are distinct. Hence $X\langle v,p\rangle$ and
$X\langle v',p'\rangle$ are disjoint.

Suppose that $(v,p)$ and $(v',p')$ are adjacent in $G\cdot H$. It
remains to prove that $X\langle v,p\rangle$ and $X\langle
v',p'\rangle$ are adjacent in \TCP{G}{H}{H}. By definition, $vv'\in
E(G)$, or $v=v'$ and $pp'\in E(H)$. If $vv'\in E(G)$, then without
loss of generality, $v$ is black and $v'$ is white, implying that
$(v,p,p')$, which is in $X\langle v,p\rangle$, is adjacent to
$(v',p,p')$, which is in $X\langle v',p'\rangle$. If $v=v'$ and
$pp'\in E(H)$, then for every $q\in V(H)$, the vertex $(v,p,q)$, which
is in $X\langle v,p\rangle$, is adjacent to $(v,p',q)$, which is in
$X\langle v',p'\rangle$.
\end{proof}

\propref{LexMinor} motivates studying $\eta(G\cdot H)$. We now show
that when $G$ is a complete graph, $\eta(G\cdot H)$ can be be
determined precisely.

\begin{proposition}
  \proplabel{LexCliqueMinor} For every graph $H$,
$$\eta(K_n\cdot H)=\FLOOR{\frac{n}{2}\big(\verts{H}+\omega(H)\big)}\enspace.$$
\end{proposition}

\begin{proof}
  Let $C$ be a maximum clique in $H$.  Consider the set of
  vertices $$X:=\{(u,y):u\in V(K_n),y\in C\}$$ in $K_n\cdot H$.  Thus
  $|X|=n\cdot\omega(H)$ and $X$ is a clique in $K_n\cdot H$. In fact,
  $X$ is a maximum clique, since every set of $n\cdot\omega(H)+1$
  vertices in $K_n\cdot H$ contains $\omega(H)+1$ vertices in a single
  copy of $H$. Thus $\omega(K_n\cdot H)=n\cdot\omega(H)$. Hence the
  upper bound on $\eta(K_n\cdot H)$ follows from
  \lemref{NewUpperBound}.

  It remains to prove the lower bound on $\eta(K_n\cdot H)$.  Delete
  the edges of $H$ that are not in $C$.  This operation is allowed
  since it does not increase $\eta(K_n\cdot H)$.  So $H$ now consists
  of $C$ and some isolated vertices.  Observe that $(K_n\cdot H)-X$ is
  isomorphic to the balanced complete $n$-partite graph with
  $\verts{H}-\omega(H)$ vertices in each colour class.  Every balanced
  complete multipartite graph with $r$ vertices has a matching of
  \floor{\frac{r}{2}} edges \citep{Sitton-EJUM}.  Thus $(K_n\cdot
  H)-X$ has a matching $M$ of \floor{\frac{n}{2}(\verts{H}-\omega(H))}
  edges.  No edge in $M$ is incident to a vertex in $X$.  For every
  edge $(v,x)(v',x')$ in $M$ and vertex $(u,y)$ of $K_n\cdot H$, since
  $v\neq v'$, without loss of generality, $v\neq u$. Thus $vu\in
  E(K_n)$ and $(v,x)$ is adjacent to $(u,y)$ in $K_n\cdot H$. Hence
  contracting each edge in $M$ gives a $K_{|X|+|M|}$-minor in
  $K_n\cdot H$. Therefore
$$\eta(K_n\cdot H) 
\geq|X|+|M|
=n\cdot\omega(H)+\FLOOR{\frac{n}{2}\big(\verts{H}-\omega(H)\big)}
=\FLOOR{\frac{n}{2}\big(\verts{H}+\omega(H)\big)}.$$
\end{proof}

We now show that \twopropref{LexMinor}{LexCliqueMinor} are closely
related to some previous results in the paper.

\propref{LexMinor} with $G=K_2$ implies that $K_2\cdot H$ is a minor
of \TCP{H}{H}{K_2}. \propref{LexCliqueMinor} implies that
$\eta(K_2\cdot H) =\verts{H}+\omega(H)$. Thus
$\eta(\TCP{H}{H}{K_2})\geq \verts{H}+\omega(H)$, which is only
slightly weaker than \propref{DoubleGridLike}.

\propref{LexMinor} with $G=K_{n,n}$ implies that $K_{n,n}\cdot H$ is a
minor of \TCP{K_{n,n}}{H}{H} for every connected graph $H$. Since
$K_{n+1}$ is a minor of $K_{n,n}$, we have $K_{n+1}\cdot H$ is a minor
of \TCP{K_{n,n}}{H}{H}. \propref{LexCliqueMinor} implies that
$$\eta(\TCP{K_{n,n}}{H}{H})
\geq\FLOOR{\frac{n+1}{2}\big(\verts{H}+\omega(H)\big)}\enspace.$$
Since $K_{n,n}\subset K_{2n}$,
$$\eta(\TCP{K_{2n}}{H}{H}) \geq\FLOOR{\frac{n+1}{2}\big(\verts{H}+\omega(H)\big)}\enspace.$$
With $H=K^d_{m}$ we have
$$\eta(\CP{K_{2n}}{K_m^{2d}}) \geq\FLOOR{\frac{n+1}{2}\big(m^d+m\big)},$$
which is equivalent to \thmref{CompleteProduct} with $n_1=n$ and
$n_2=\dots=n_{2d+1}=m$ (ignoring lower order terms).  In fact for
small values of $m$, this bound is stronger than
\thmref{CompleteProduct}. For example, with $n=1$ and $m=3$ we have
$$\eta(\CP{K_{2}}{K_3^{2d}}) \geq 3^d+3,$$
whereas \thmref{CompleteProduct} gives no non-trivial bound on
$\eta(\CP{K_{2}}{K_3^{2d}})$.

\section{Rough Structural Characterisation Theorem for Trees}
\seclabel{TwoTrees}

In this section we characterise when the product of two trees has a
large clique minor. \secref{StarMinors} gives such an example:
\twocorref{GraphStarMinor}{StarMinorTree} imply that if one tree has
many leaves and the other has many vertices then their product has a
large clique minor. Now we give a different example. As illustrated in
\figref{Balance}(a), let $B_n$ be the tree obtained from the path
$P_{2n+1}$ by adding one leaf adjacent to the vertex in the middle of
$P_{2n+1}$.

\Figure{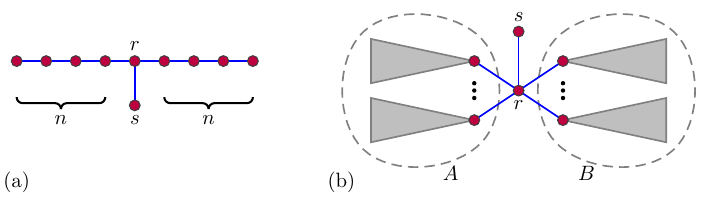}{(a) The tree $B_n$. (b) A balance of order
  $\min\{|A|,|B|\}$.}

Now $B_n$ only has three leaves, but \citet{SW-JCTB89} implicitly
proved that the product of $B_n$ and a long path (which only has two
leaves) has a large clique minor\footnote{\citet{SW-JCTB89} observed
  that since the complete graph has a drawing in the plane with all
  the crossings collinear, \CP{B_n}{P_m} contains clique subdivisions
  of unbounded order, and thus $\eta(\CP{B_n}{P_m})$ is unbounded. On
  the other hand, \citet{SW-JCTB89} proved that if $T$ is the tree
  obtained from the path $(v_1,\dots,v_m)$ by adding one leaf adjacent
  to $v_2$, then $\eta(\CP{P_n}{T})\leq 7$. This observation disproved
  an early conjecture by Robertson and Seymour about the structure of
  graphs with an excluded minor, and lead to the development of
  vortices in Robertson and Seymour's theory; see
  \citep{RS-GraphMinorsXVI-JCTB03,KM-GC07}.}. In \thmref{BigBalance}
below we prove an explicit bound of
$\eta(\CP{P_m}{B_n})\geq\min\{n,\sqrt{m}\}$, which is illustrated in
\figref{TennisCourt}. In fact, this bound holds for a more general
class of trees than $B_n$, which we now introduce.

\Figure{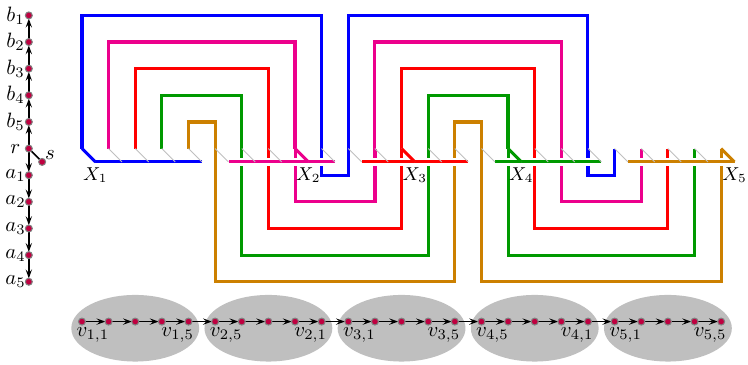}{$K_n$ is a minor of \CP{P_{n^2}}{B_n}.}

As illustrated in \figref{Balance}(b), a \emph{balance of order} $n$
is a tree $T$ that has an edge $rs$, and disjoint sets $A,B\subseteq
V(T)-\{r,s\}$, each with at least $n$ vertices, such that $A\cup\{r\}$
and $B\cup\{r\}$ induce connected subtrees in $T$. We say $A$ and $B$
are the \emph{branches}, $r$ is the \emph{root}, and $s$ is the
\emph{support} of $T$. For example, the star $S_t$ is a balance of
order $\floor{\tfrac{t-1}{2}}$, and $B_n$ is a balance of order
$n$. \thmref{BigBalance} below implies that
$\eta(\CP{P_m}{T})\geq\min\{\sqrt{m},n\}$ for every balance $T$ of
order $n$. The critical property of a long path is that it has many
large disjoint subpaths.

\begin{theorem}
  \thmlabel{BigBalance} Let $G$ be a tree that has $n$ disjoint
  subtrees each of order at least $n$. Then for every balance $T$ of
  order $n$, $$\eta(\CP{G}{T})\geq n\enspace.$$
\end{theorem}

\begin{proof}
  Let $A$ and $B$ be the branches, let $r$ be the root, and let $s$ be
  the support of $T$. Contract edges in $T$ until $A$ and $B$ each
  have exactly $n$ vertices. Orient the edges of $T$ away from
  $r$. Label the vertices of $A$ by $\{a_1,a_2,\dots,a_n\}$ such that
  if $\overrightarrow{a_ia_j}\in E(A)$ then $i<j$. Label the vertices
  of $B$ by $\{b_1,b_2,\dots,b_n\}$ such that if
  $\overrightarrow{b_ib_j}\in E(B)$ then $j<i$. For each $i\in[n]$,
  let $T_i$ be the path between $a_i$ and $b_i$ in $T$ (which thus
  includes $r$).

  Contract edges in $G$ until it is the union of $n$ disjoint subtrees
  $G_1,\dots,G_n$, each with exactly $n$ vertices. For every pair of
  vertices $v,w\in V(G)$, let $G\langle v,w\rangle$ be the path
  between $v$ and $w$ in $G$. Orient the edges of $G$ away from an
  arbitrary vertex in $G_1$. Let $G^*$ be the oriented tree obtained
  from $G$ by contracting each $G_i$ into a single vertex $z_i$. Note
  that for each $i\in[2,n]$, each vertex $z_i$ has exactly one
  incoming arc in $G^*$. Fix a proper $2$-colouring of $G^*$ with
  colours \emph{black} and \emph{white}, where $z_1$ is coloured
  white. Label the vertices in each subtree $G_i$ by
  $\{v_{i,1},\dots,v_{i,n}\}$, such that for each arc
  $(v_{i,j},v_{i,\ell})$ in $G_i$, we have $\ell<j$ if $z_i$ is black,
  and $j<\ell$ if $z_i$ is white.

  For each $i\in[n]$, let $H_i$ be the subgraph of \CP{G}{T} induced
  by $$\{(v_{i,j},s):j\in[n]\}\enspace.$$ For all $i,j\in[n]$, let
  $T_{i,j}$ be the subgraph of \CP{G}{T} induced
  by $$\{(v_{j,i},y):y\in V(T_i)\}\enspace.$$

  For all $i,j\in[n]$, let $c_{i,j}$ be the vertex $a_i$ if $z_j$ is
  white, and $b_i$ if $z_j$ is black.  For all $i\in[n]$ and
  $j\in[2,n]$, let $U_{i,j}$ be the subgraph of \CP{G}{T} induced by
  \begin{align*}
    \{(x,c_{i,j}):x\in G\langle v_{k,i},v_{j,i}\rangle\},
  \end{align*}
  where $(z_k,z_j)$ is the incoming arc at $z_j$ in $G^*$.

  For each $i\in[n]$, let $X_i$ be the subgraph of \CP{G}{T} induced
  by
$$\bigcup_{j\in[n]}(T_{i,j}\cup U_{i,j}\cup H_i)\enspace.$$
We now prove that $X_1,\dots,X_n$ are the branch sets of a
$K_n$-minor.

First we prove that each $X_i$ is connected.  Observe that each $H_i$
is isomorphic to $G_i$, and is thus connected.  Each $T_{i,j}$ is
isomorphic to the path $T_i$, and is thus connected.  Moreover, the
endpoints of $T_{i,j}$ are $(v_{j,i},a_i)$ and $(v_{j,i},b_i)$.  Each
$U_{i,j}$ is isomorphic to the path $G\langle
v_{j',i},v_{j,i}\rangle$, and is thus connected.  Moreover, the
endpoints of $U_{i,j}$ are $(v_{j',i},c_{i,j})$ and
$(v_{j,i},c_{i,j})$.  Thus if $z_j$ is white (and thus $z_{j'}$ is
black), then the endpoints of $U_{i,j}$ are $(v_{j',i},a_i)$ and
$(v_{j,i},a_i)$.  And if $z_j$ is black (and thus $z_{j'}$ is white),
then the endpoints of $U_{i,j}$ are $(v_{j',i},b_i)$ and
$(v_{j,i},b_i)$.  Hence the induced subgraph
$$T_{i,1}\cup U_{i,1}\cup T_{i,2}\cup U_{i,2}\cup T_{i,3}\cup U_{i,3}\cup \dots\cup T_{i,n}\cup U_{i,n}$$
is a path (illustrated in \figref{TennisCourt} by alternating vertical
and horizontal segments).  Furthermore, the vertex $(v_{i,i},s)$ in
$H_i$ is adjacent to the vertex $(v_{i,i},r)$ in $T_{i,i}$. Thus $X_i$
is connected.

Now we prove that the subgraphs $X_i$ and $X_{i'}$ are disjoint for
all distinct $i,i'\in[n]$.  First observe that $H_i$ and $H_{i'}$ are
disjoint since the first coordinate of every vertex in $H_i$ is some
$v_{i,j}$.  Similarly, for all $j,j'\in[n]$, the subgraphs $T_{i,j}$
and $T_{i',j'}$ are disjoint since the first coordinate of every
vertex in $T_{i,j}$ is $v_{j,i}$.  For all $j,j'\in[n]$, the subgraphs
$U_{i,j}$ and $U_{i',j'}$ are disjoint since the second coordinate of
every vertex in $U_{i,j}$ is $a_i$ or $b_i$.  For all $j\in[n]$, $H_i$
is disjoint from $T_{i',j}\cup U_{i',j}$ since the second coordinate
of $H_i$ is $s$. It remains to prove that $T_{i,j}$ and $U_{i',j'}$
are disjoint. Suppose on the contrary that for some $j,j'\in[n]$, some
vertex $(x,y)$ is in $T_{i,j}\cap U_{i',j'}$. Without loss of
generality, $z_{j'}$ is black.  Say $(z_k,z_{j'})$ is the incoming arc
at $z_{j'}$ in $G^*$. So $z_k$ is white. Since $(x,y)\in T_{i,j}$, we
have $x=v_{j,i}$ and $y\in V(T_i)$. Since $(x,y)\in U_{i',j'}$, $x$ is
in the path $G\langle v_{k,i'},v_{j',i'}\rangle$. Since $z_{j'}$ is
black, $y=b_{i'}$. Now $v_{j,i}$ (which equals $x$) is in the path
$G\langle v_{k,i'},v_{j',i'}\rangle$. Thus by the labelling of
vertices in $G_k$ and $G_{j'}$, we have $i'<i\leq n$. By comparing the
second coordinates, observe that $b_{i'}\in V(T_i)$. Thus $i'>i$ by
the labelling of the vertices in $B$. This contradiction proves that
$T_{i,j}$ and $U_{i',j'}$ are disjoint for all $j,j'\in[n]$. Hence
$X_i$ and $X_{i'}$ are disjoint.

Finally, observe that for distinct $i,i'\in[n]$, the vertex
$(v_{i,i'},s)$ in $H_i\subset X_i$ is adjacent to the vertex
$(v_{i,i'},r)$ in $T_{i',i}\subset X_{i'}$. Therefore the $X_i$ are
branch sets of a $K_n$-minor.
\end{proof}

We conjecture that the construction in \thmref{BigBalance} is within a
constant factor of optimal; that is,
$\CP{P_m}{B_n}=\Theta(\min\{\sqrt{m},n\})$.


\thmref{BigBalance} motivates studying large disjoint subtrees in a
given tree.  Observe that a star does not have two disjoint subtrees,
both with at least two vertices. Thus a star cannot be used as the
tree $G$ in \thmref{BigBalance} with $n\geq2$. On the other hand, a
path on $n^2$ vertices has $n$ disjoint subpaths, each with $n$
vertices. Of the trees with the same number of vertices, the star has
the most leaves and the path has the least. We now prove that the
every tree with few leaves has many large disjoint
subtrees\footnote{As an aside we now describe a polynomial-time
  algorithm that for a given tree $T$, finds the maximum number of
  disjoint subtrees in $T$, each with at least $n$ vertices. It is
  convenient to consider a generalisation of this problem, where each
  vertex $v$ is assigned a positive weight $w(v)$, and each subtree is
  required to have total weight at least $n$. Let $v$ be a leaf of
  $T$. Let $T':=T-v$. Define a new weight function $w'(z):=w(z)$ for
  every vertex $z$ of $T-v$. First suppose that $w(v)\geq n$. Then $T$
  has $k$ disjoint subtrees each of total $w$-weight at least $n$ if
  and only if $T'$ has $k-1$ disjoint subtrees each of total
  $w'$-weight at least $n$. (In which case $T[\{v\}]$ becomes one of
  the $k$ subtrees.)\ Now assume that $w(v)<n$. Let $x$ be the
  neighbour of $v$ in $T$. Redefine $w'(x):=w(x)+w(v)$. Then $T$ has
  $k$ disjoint subtrees each of total $w$-weight at least $n$ if and
  only if $T'$ has $k$ disjoint subtrees each of total $w'$-weight at
  least $n$. (A subtree $X$ of $T'$ containing $x$ is replaced by the
  subtree $T[V(X)\cup\{v\}]$.)\ Thus in each case, from an inductively
  computed optimal solution in $T'$ (for the weight function $w'$), we
  can compute an optimal solution in $T$ (for the weight function
  $w$).}.

\begin{theorem}
  \thmlabel{TreePartition} Let $T$ be a tree with at least one edge,
  and let $n$ be a positive integer, such that $$\verts{T}\geq
  n^2+(\STAR{T}-2)(n-1)+1\enspace.$$ Then $T$ has $n$ disjoint
  subtrees, each with at least $n$ vertices.
\end{theorem}

The proof of \thmref{TreePartition} is based on the following lemma.

\begin{lemma}
  \lemlabel{WeightedTreePartition} Fix a positive integer $n$.  Let
  $T$ be a tree with at least one edge, where each vertex of $T$ is
  assigned a weight $w(v)\in\mathbb{Z}^+$, such that every leaf has
  weight at most $n$ and every other vertex has weight $1$.  Let
  $W(T):=\sum_{v\in V(T)}w(v)$ be the total weight.  Then there is a
  vertex-partition of $T$ into at least
$$f(T):=\FLOOR{\frac{W(T)-(\STAR{T}-2)(n-1)}{n}}$$ disjoint subtrees, 
each with total weight at least $n$.
\end{lemma}

\begin{proof}
  We proceed by induction on $f(T)$. If $f(T)\leq 0$ then there is
  nothing to prove. First suppose that $f(T)=1$. Then $W(T)\geq
  (\STAR{T}-2)(n-1)+n\geq n$ since $\STAR{T}\geq2$. Thus $T$ itself
  has total weight at least $n$, and we are done.  Now assume that
  $f(T)\geq 2$.

  Suppose that $T=K_2$ with vertices $x$ and $y$, both of which are
  leaves. Thus $\STAR{T}=2$ and
  $f(T)=\floor{\frac{w(x)+w(y)}{n}}$. Now $f(T)\geq 2$ and
  $w(x),w(y)\leq n$. Thus $f(T)=2$ and $w(x)=w(y)=n$. Hence $T[\{x\}]$
  and $T[\{y\}]$ is a vertex-partition of $T$ into two disjoint
  subtrees, each with weight at least $n$, and we are done. Now assume
  that $\verts{T}\geq3$.

  Suppose that $w(v)=n$ for some leaf $v$. Let $T':= T-v$. By
  induction, there is a vertex-partition of $T'$ into $f(T')$ disjoint
  subtrees, each with total weight at least $n$. These subtrees plus
  $T[\{v\}]$ are a vertex-partition of $T$ into $1+f(T')$ disjoint
  subtrees, each with total weight at least $n$.  Now $W(T')=W(T)-n$,
  and $s(T')\leq \STAR{T}$ since $v$ is not a leaf in $T'$, and the
  neighbour of $v$ is the only potential leaf in $T'$ that is not a
  leaf in $T$. Thus
  \begin{align*}
    1+f(T')   =\;&\FLOOR{\frac{n+W(T')-(s(T')-2)(n-1)}{n}} \\
    \geq\;&  \FLOOR{\frac{W(T)-(\STAR{T}-2)(n-1)}{n}} \\
    =\;&f(T)\enspace,
  \end{align*}
  and we are done. Now assume that $w(v)\leq n-1$ for every leaf $v$.

  Let $T'$ be the tree obtained from $T$ by deleting each leaf. Since
  $T\ne K_2$, $T'$ has at least one vertex. Suppose that $T'$ has
  exactly one vertex. That is, $T$ is a star. Thus $W(T)\leq
  1+\STAR{T}\cdot(n-1)$ since every leaf has weight at most
  $n-1$. Thus $$f(T)=\FLOOR{\frac{W(T)-(\STAR{T}-2)(n-1)}{n}}\leq\FLOOR{\frac{2n-1}{n}}=1\enspace,$$
  which is a contradiction.

  Now assume that $T'$ has at least two vertices. In particular, $T'$
  has a leaf $v$. Note that the neighbour of $v$ in $T'$ is not a leaf
  in $T$, as otherwise $T'$ would only have one vertex. Now $v$ is not
  a leaf in $T$ (since it is in $T'$). Thus $v$ is adjacent to at
  least one leaf in $T$. Let $x_1,x_2,\dots,x_d$ be the leaves of $T$
  that are adjacent to $v$.

  First suppose that $\sum_iw(x_i)\leq n-1$. Let
  $T'':=T-\{x_1,\dots,x_d\}$. Then $v$ is a leaf in $T''$. Redefine
  $w(v):=1+\sum_iw(x_i)$. By induction, there is a vertex-partition of
  $T''$ into $f(T'')$ disjoint subtrees, each of total weight at least
  $n$. Observe that $W(T)=W(T'')$ and $s(T'')=\STAR{T}-d+1\leq
  \STAR{T}$. Thus
  \begin{align*}
    f(T'') =\FLOOR{\frac{W(T'')-(s(T'')-2)(n-1)}{n}}
    \geq\;&\FLOOR{\frac{W(T)-(\STAR{T}-2)(n-1)}{n}}\\
    =\;&f(T)\enspace.\end{align*} Adding $x_1,\dots,x_d$ to the
  subtree of $T''$ containing $v$ gives a vertex-partition of $T$ into
  at least $f(T)$ disjoint subtrees, each with total weight at least
  $n$.

  Now assume that $\sum_iw(x_i)\geq n$. Let
  $T''':=T-\{v,x_1,\dots,x_d\}$. Now $s(T''')\leq \STAR{T}-d+1$, since
  $w_1,\dots,w_d$ are leaves in $T$ that are not in $T'''$, and the
  neighbour of $v$ in $T''$ is the only vertex in $T'''$ that possibly
  is a leaf in $T'''$ but not in $T$. Observe that $W(T''')\geq
  W(T)-1-d(n-1)$ since $v$ has weight $1$ and each $w_i$ has weight at
  most $n-1$. By induction, there is partition of $V(T''')$ into at
  least $f(T''')$ disjoint subtrees, each with total weight at least
  $n$. These subtrees plus $T[\{v,x_1,\dots,x_d\}]$ are a
  vertex-partition of $T$ into at least $1+f(T''')$ disjoint subtrees,
  each with total weight at least $n$. Now
  \begin{align*}
    1+f(T''')
    &=1+\FLOOR{\frac{W(T''')-(s(T''')-2)(n-1)}{n}}\\
    &\geq\FLOOR{\frac{n+W(T)-1-d(n-1)-(\STAR{T}-d+1-2)(n-1)}{n}}\\
    &=\FLOOR{\frac{W(T)-(\STAR{T}-2)(n-1)}{n}}\\
    &=f(T)\enspace.
  \end{align*}
  This completes the proof.
\end{proof}

\begin{lemma}
  \lemlabel{TreePartition} For every positive integer $n$, every tree
  $T$ with at least one edge has
$$\FLOOR{\frac{\verts{T}-(\STAR{T}-2)(n-1)}{n}}$$ disjoint subtrees, 
each with at least $n$ vertices. Moreover, for all integers $s,n\geq2$
and $N\geq sn$, there is a tree $T$ with $\verts{T}=N$ and
$\STAR{T}=s$, such that $T$ has at most
$$\FLOOR{\frac{\verts{T}-(\STAR{T}-2)(n-1)}{n}}$$ disjoint subtrees, each with at least $n$ vertices. 
\end{lemma}

\begin{proof}
  \lemref{WeightedTreePartition} with each leaf assigned a weight of
  $1$ implies the first claim. It remains to construct the tree
  $T$. Fix a path $P$ with $N-s(n-1)$ vertices. Let $v$ and $w$ be the
  endpoints of $P$. As illustrated in \figref{PartitionExample}, let
  $T$ be the tree obtained from $P$ by attaching $\ceil{\frac{s}{2}}$
  pendant paths to $v$, each with $n-1$ vertices, and by attaching
  $\floor{\frac{s}{2}}$ pendant paths to $w$, each with $n-1$
  vertices. Thus $T$ has $N$ vertices and $s$ leaves.

  \Figure{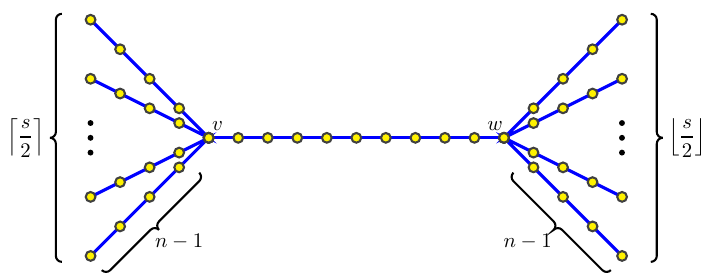}{The tree $T$ in \lemref{TreePartition}.}

  Let $A$ be the subtree of $T$ induced by the union of $v$ and the
  pendant paths attached at $v$. Let $B$ be the subtree of $T$ induced
  by the union of $w$ and the pendant paths attached at $w$. Let
  $X_1,\dots,X_t$ be a set of disjoint subtrees in $T$, each with at
  least $n$ vertices. If some $X_i$ intersects $A$ then $v$ is in
  $X_i$. At most one subtree $X_i$ contains $v$. Thus at most one
  subtree $X_i$ intersects $A$. Similarly, at most one subtree $X_j$
  intersects $B$. The remaining $t-2$ subtrees are contained in
  $P-\{v,w\}$. Thus $$t-2\leq\frac{N-s(n-1)-2}{n}\enspace.$$ It
  follows that $tn\leq N-(s-2)(n-1)$, as desired.
\end{proof}

\begin{proof}[Proof of \thmref{TreePartition}]
  The assumption in \thmref{TreePartition} implies that
$$\frac{\verts{T}-(\STAR{T}-2)(n-1)}{n}\geq n+\frac{1}{n}\enspace.$$
Hence
$$\FLOOR{\frac{\verts{T}-(\STAR{T}-2)(n-1)}{n}}\geq n\enspace.$$
The result thus follows from \lemref{TreePartition}.
\end{proof}

For every tree $T$, let \bal{T} be the maximum order of a balance
subtree in $T$. The next lemma shows how to construct a balance of
large order. For each vertex $r$ of $T$, let $T_r$ be the component of
$T-r$ with the maximum number of vertices.

\begin{lemma}
  \lemlabel{BigBalance} Let $r$ be a vertex in a tree $T$ with degree
  $d\geq3$.  Then every balance in $T$ rooted at $r$ has order at most
  $\verts{T}-\verts{T_r}-2$.  On the other hand, there is a balance in
  $T$ rooted at $r$ of order at least
  $\third(\verts{T}-\verts{T_r}-1)$.
\end{lemma}

\begin{proof}
  Consider a balance rooted at $r$. The largest component of $T-r$ is
  contained in at most one branch. Thus the other branch has at most
  $\verts{T}-\verts{T_r}-2$ vertices. Hence the order of the balance
  is at most $\verts{T}-\verts{T_r}-2$.

  Now we prove the second claim. Let $Y_1,\dots,Y_d$ be the components
  of $T-r$. Say $\verts{Y_1}\geq\dots\geq \verts{Y_d}$. Then
  $\verts{T_r}=\verts{Y_1}$. Let $s$ be the neighbour of $r$ in $Y_d$.

  First suppose that $d=3$. Then $Y_1$ and $Y_2$ are the branches of a
  balance rooted at $r$ with support $s$, and order \verts{Y_2}. Now
  $2\verts{Y_2}\geq
  \verts{Y_2}+\verts{Y_3}=\verts{T}-\verts{Y_1}-1$. Thus the order of
  the balance, \verts{Y_2}, is at least
  $\half(\verts{T}-\verts{Y_1}-1)>\third(\verts{T}-\verts{Y_1}-1)$.

  Now assume that $d\geq 4$. If $A$ and $B$ are a partition of
  $[d-1]$, then $\cup\{Y_i:i\in A\}$ and $\cup\{Y_i:i\in B\}$ are the
  branches of a balance rooted at $r$ with support $s$. The order of
  the balance is
$$\min\SET{\sum_{i\in B}\verts{Y_i},\sum_{i\in A}\verts{Y_i}}\enspace.$$
A greedy algorithm\footnote{Given integers $m_1\geq m_2\geq\dots\geq
  m_t\geq 1$ that sum to $m$, construct a partition of $[t]$ into sets
  $A$ and $B$ as follows. For $i\in[t]$, let $A_i:=\sum_{j\in[i]\cap
    A}m_j$ and $B_i:=\sum_{j\in[i]\cap B}m_j$. Initialise $A:=\{m_1\}$
  and $B:=\emptyset$. Then for $i=2,3,\dots,t$, if $A_{i-1}\leq
  B_{i-1}$ then add $i$ to $A$; otherwise add $i$ to $B$. Thus
  $|A_1-B_1|=m_1$ and by induction,
  $|A_i-B_i|\leq\max\{|A_{i-1}-B_{i-1}|-m_i,m_i\}\leq\max\{m_1-m_i,m_i\}\leq
  m_1$.  Thus $|A_t-B_t|\leq m_1$. Hence $A_t$ and $B_t$ are both at
  least $\half(m-m_1)$.} gives such a partition with
\begin{align*}
  \min\SET{\sum_{i\in B}\verts{Y_i},\sum_{i\in A}\verts{Y_i}}
  \geq\;&\half\big(-\verts{Y_1}+\sum_{i\in[d-1]}\verts{Y_i}\big)\\
  =\;&\half(\verts{T}-\verts{Y_1}-\verts{Y_d}-1)\enspace.
\end{align*}
Since $Y_d$ is the smallest of $Y_2,\dots,Y_d$, the order of the
balance is at least
\begin{align*}
  \half(\verts{T}-\verts{Y_1}-\tfrac{\verts{T}-\verts{Y_1}-1}{d-1}-1)
  =\;&\tfrac{d-2}{2(d-1)}(\verts{T}-\verts{Y_1}-1)\\
  \geq\;&\third(\verts{T}-\verts{Y_1}-1)\enspace,
\end{align*}
as desired.
\end{proof}

Which trees $T$ have small \bal{T}? First note that $\bal{P}=0$ for
every path $P$. A path $P$ in a tree $T$ is \emph{clean} if every
internal vertex of $P$ has degree $2$ in $T$. Let $p(T)$ be the
maximum number of vertices in a clean path in $T$. The \emph{hangover}
of $T$, denoted by \hang{T}, is the minimum, taken over all clean
paths $P$ in $T$, of the maximum number of vertices in a component of
$T-E(P)$. We now prove that \bal{T} and \hang{T} are tied.

\begin{lemma}
  \lemlabel{BalanceHang} For every tree $T$,
$$\bal{T}+1\leq \hang{T}\leq 3\,\bal{T}+1\enspace.$$
\end{lemma}

\begin{proof}
  First we prove the lower bound. Let $P$ be a longest clean path in
  $T$. Since every internal vertex in $P$ has degree $2$ in $T$, every
  balance in $T$ is rooted at a vertex $r$ in one of the components of
  the forest obtained by deleting the internal vertices and edges of
  $P$ from $T$. For every such vertex $r$, $T-r$ has a component of at
  least $\verts{T}-\hang{T}-1$ vertices. That is, $\verts{T_r}\geq
  \verts{T}-\hang{T}-1$. By \lemref{BigBalance}, every balance rooted
  at $r$ has order at most
  $\verts{T}-\verts{T_r}-2\leq\hang{T}-1$. Thus
  $\bal{T}\leq\hang{T}-1$.

  Now we prove the upper bound. If $T$ is a path then
  $\bal{T}=\hang{T}=0$, and we are done. Now assume that $T$ has a
  vertex of degree at least $3$. Let $r$ be a vertex of degree at
  least $3$ in $T$ such that $\verts{T_r}$ is minimised. Let $x$ be
  the closest vertex in $T_r$ to $r$ such that $\deg(x)\neq 2$. Let
  $P$ be the path between $r$ and $x$ in $T$. Thus $P$ is clean. If
  $x$ is a leaf, then $\verts{T_x}=\verts{T}-1\geq \verts{T_r}$. If
  $\deg(x)\geq3$ then $\verts{T_x}\geq \verts{T_r}$ by the choice of
  $r$. In both cases $\verts{T_x}\geq \verts{T_r}$, which implies that
  $r$ is in $T_x$. Thus deleting the internal vertices and edges of
  $P$ gives a forest with two components, one with
  $\verts{T}-\verts{T_r}$ vertices, and the other with
  $\verts{T}-\verts{T_x}$ vertices. Hence
  $\hang{T}\leq\max\{\verts{T}-\verts{T_r},\verts{T}-\verts{T_x}\}=\verts{T}-\verts{T_r}$. By
  \lemref{BigBalance}, there is a balance in $T$ rooted at $r$ of
  order at least
  $\third(\verts{T}-\verts{T_r}-1)\geq\third(\hang{T}-1)$. Hence
  $\bal{T}\geq\third(\hang{T}-1)$.
\end{proof}

We now prove that if the product of sufficiently large trees has
bounded Hadwiger number then both trees have bounded hangover.

\begin{lemma}
  \lemlabel{BigTrees} Fix an integer $c\geq1$. Let $T_1$ and $T_2$ be
  trees, such that $\verts{T_1}\geq2c^2-c+2$, $\verts{T_2}\geq c+1$,
  and $\eta(\CP{T_1}{T_2})\leq c$. Then $$\hang{T_2}\leq
  3c+1\enspace.$$ By symmetry, if in addition $\verts{T_2}\geq
  2c^2-c+2$ then
$$\hang{T_1}\leq 3c+1\enspace.$$
\end{lemma}

\begin{proof}
  If $\STAR{T_1}\geq c$, then by \corref{GraphStarMinor},
  $\eta(\CP{T_1}{T_2})\geq \min\{\verts{T_2},\STAR{T_1}+1\}\geq c+1$,
  which contradicts the assumption. Now assume that $\STAR{T_1}\leq
  c-1$.

  Let $n:=c+1$. Then $\verts{T_1}\geq (c+1)^2+(c-3)c+1\geq
  n^2+(\STAR{T_1}-2)(n-1)+1$. Thus \thmref{TreePartition} is
  applicable to $T_1$ with $n=c+1$.  Hence $T_1$ has $c+1$ disjoint
  subtrees, each with at least $c+1$ vertices.

  If $\bal{T_2}\geq c+1$, then by \thmref{BigBalance},
  $\eta(\CP{T_1}{T_2})\geq\min\{c+1,\bal{T_2}\}=c+1$, which
  contradicts the assumption.  Thus $\bal{T_2}\leq c$, and by
  \lemref{BalanceHang}, $\hang{T_2}\leq 3c+1$.
\end{proof}

We now prove a converse result to \lemref{BigTrees}. It says that the
product of two trees has small Hadwiger number whenever one of the
trees is small or both trees have small hangover.

\begin{lemma}
  \lemlabel{RoughTree} Let $T_1$ and $T_2$ be trees, such that for
  some integer $c\geq1$,
  \begin{itemize}
  \item $\verts{T_1}\leq c$ or $\verts{T_2}\leq c$, or
  \item $\hang{T_1}\leq c$ and $\hang{T_2}\leq c$.
  \end{itemize}
  Then $\eta(\CP{T_1}{T_2})\leq c'$ for some $c'$ depending only on
  $c$.
\end{lemma}

\begin{proof}
  First suppose that $\verts{T_1}\leq c$. Then
  $\eta(\CP{T_1}{T_2})\leq\eta(\CP{K_{c}}{T_2})=c+1$ by
  \thmref{TreeComplete} below. Similarly, if $\verts{T_2}\leq c$ then
  $\eta(\CP{T_1}{T_2})\leq c+1$.

  Otherwise, by assumption, $\hang{T_1}\leq c$ and $\hang{T_2}\leq c$.
  For $i\in[2]$, let $P_i$ be a clean path with $p(T_i)$ vertices in
  $T_i$.  Now \CP{P_1}{P_2} is a planar $p(T_1)\times p(T_2)$ grid
  subgraph $H$ in \CP{T_1}{T_2}. We now show that \CP{T_1}{T_2} can be
  obtained from $H$ by adding a vortex in the outerface of $H$.

  Let $F$ be the set of vertices on the outerface of $H$ in clockwise
  order.  Consider a vertex $v=(v_1,v_2)\in F$, where $v_1\in V(P_1)$
  and $v_2\in V(P_2)$.  For $i\in[2]$, let $A_i(v)$ be the component
  of $T_i-E(P_i)$ that contains $v_i$.  As illustrated in
  \figref{GridVortex}, define $S(v)$ to be the set $\{(a,b):a\in
  A_1(v),b\in A_2(v)\}$ of vertices in \CP{T_1}{T_2}.  Every vertex of
  $G-H$ is in $S(v)$ for some vertex $v\in F$.  In addition, each
  vertex $v\in F$ is in $S(v)$.  For every edge $e$ of \CP{T_1}{T_2}
  where both endpoints of $e$ are in $\cup\{S(v):v\in F\}$, the
  endpoints of $e$ are in one bag or in bags corresponding to
  consecutive vertices in $F$.  For each vertex $v\in F$, if $w$ is
  clockwise from $v$ in $F$, then define $S'(v):=S(v)\cup S(w)$.
  Hence for every edge $xy$ of \CP{T_1}{T_2} where both endpoints of
  $e$ are in $\cup\{S(v):v\in F\}$, the endpoints of $e$ are both in
  $S'(v)$ for some vertex $v\in F$.  Now
  $|S(v)|=|A_1(v)|\cdot|A_2(v)|\leq\hang{T_1}^2\leq c^2$.  Thus
  $\{S'(v):v\in F\}$ is a vortex of width at most $2c^2$.

  It is well known that every graph obtained from a graph embedded in
  a surface of bounded genus by adding a vortex of bounded width has
  bounded Hadwiger number.   Thus $\eta(\CP{T_1}{T_2})$ is at most some constant depending only
  on $c$. Recently,  \citet{CliqueMinors} proved a tight bound on the
  Hadwiger number of such a graph; it implies that
  $\eta(\CP{T_1}{T_2})\leq \Oh{c^2}$. \end{proof}

\Figure{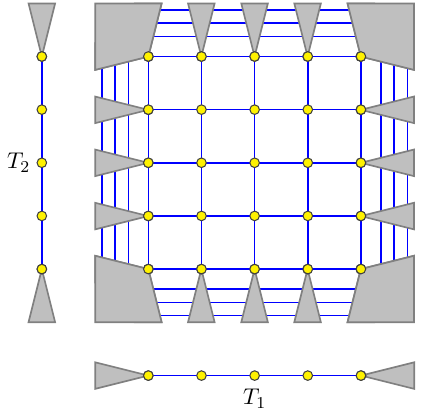}{The sets $S(v)$ in the proof of
  \lemref{RoughTree}.}


\twolemref{BigTrees}{RoughTree} imply the following rough structural
characterisation theorem for the products of trees.

\begin{theorem}
  \thmlabel{RoughTree} For trees $T_1$ and $T_2$ each with at least
  one edge, the function $\eta(\CP{T_1}{T_2})$ is tied to $\min\big\{
  \verts{T_1},\, \verts{T_2},\, \max\{\hang{T_1},\hang{T_2}\} \,\big\}
  \enspace.$
\end{theorem}

\thmref{RoughTree} can be informally stated as: $\eta(\CP{T_1}{T_2})$
is bounded if and only if:
\begin{itemize}
\item $\verts{T_1}$ or $\verts{T_2}$ is bounded, or
\item $\hang{T_1}$ and $\hang{T_2}$ are bounded.
\end{itemize}
\thmref{RoughTree} is generalised for products of arbitrary graphs in
\thmref{RoughGraph} below.

\section{Product of a General Graph and a Complete Graph}
\seclabel{GeneralComplete}

This section studies the Hadwiger number of the product of a general
graph and a complete graph. \citet{Miller-DM78} stated without proof
that $\eta(\CP{T}{K_n})=n+1$ for every tree $T$ and integer
$n\geq2$. We now prove this claim.

\begin{theorem}
  \thmlabel{TreeComplete} For every tree $T$ with at least one edge
  and integer $n\geq1$,
$$\eta(\CP{T}{K_n})=n+1\enspace.$$
\end{theorem}

\begin{proof}
  Since \CP{K_2}{K_n} is a subgraph of \CP{T}{K_n}, the lower bound
  $\eta(\CP{T}{K_n})\geq n+1$ follows from \propref{CompleteEdge}. It
  remains to prove the upper bound $\eta(\CP{T}{K_n})\leq n+1$. Let
  $X_1,\dots,X_k$ be the branch sets of a complete minor in
  \CP{T}{K_n}, where $k=\eta(\CP{T}{K_n})$. For each $i\in[k]$, let
  $T_i$ be the subtree of $T$ consisting of the edges $vw\in E(T)$
  such that $(v,j)$ or $(w,j)$ is in $X_i$ for some $j\in[n]$. Since
  $X_i$ is connected, $T_i$ is connected. Since $X_i$ and $X_j$ are
  adjacent, $T_i$ and $T_j$ share an edge in common. By the Helly
  property of trees, there is an edge $vw$ of $T$ in every subtree
  $T_i$. Let $Y$ be the set of vertices
  $\{(v,j),(w,j):j\in[n]\}$. Thus, by construction, every $X_i$
  contains a vertex in $Y$. Since $|Y|=2n$, if every $X_i$ has at
  least two vertices in $Y$, then $k\leq n$, and we are done. Now
  assume that some $X_i$ has only one vertex in $Y$. Say $(v,j)$ is
  the vertex in $X_i\cap Y$. Let $T_v$ and $T_w$ be the subtrees of
  $T$ obtained by deleting the edge $vw$, where $v\in V(T_v)$ and
  $w\in V(T_w)$. Thus $X_i$ is contained in \CP{T_v}{K_n}. Let $Z$ be
  the set of neighbours of $(v,j)$ in $Y$. That is,
  $Z=\{(v,\ell):\ell\in[n]-\{j\}\}\cup\{(w,j)\}$. Suppose on the
  contrary that some branch set $X_p$ ($p\ne i$) has no vertex in
  $Z$. Then $X_p$ is contained in \CP{T_w}{K_n} minus the vertex
  $(w,j)$. Thus $X_i$ and $X_p$ are not adjacent. This contradiction
  proves that every branch set $X_p$ ($p\ne i$) has a vertex in
  $Z$. Since $|Z|=n$, $k\leq n+1$, as desired.
\end{proof}

\thmref{TreeComplete} is generalised through the notion of treewidth.

\begin{theorem}
  \thmlabel{TreewidthComplete} For every graph $G$ and integer
  $n\geq1$,
$$\eta(\CP{G}{K_n})\leq\tw{\CP{G}{K_n}}+1\leq n(\tw{G}+1)\enspace.$$
Moreover, for all integers $k\geq2$ and $n\geq2$ there is a graph $G$
with $\tw{G}=k$
and $$\eta(\CP{G}{K_n})=\tw{\CP{G}{K_n}}+1=n(k+1)\enspace.$$
\end{theorem}

\begin{proof}
  First we prove the upper bound\footnote{This upper bound even holds
    for strong products.}. Let $(T,\{T_x\subseteq V(G):x\in V(T)\})$
  be a tree decomposition of $G$ with at most $\tw{G}+1$ vertices in
  each bag. Replace each bag $T_x$ by $\{(v,i):v\in T_x,i\in[n]\}$. We
  obtain a tree decomposition of \CP{G}{K_n} with at most
  $n(\tw{G}+1)$ vertices in each bag. Thus $\tw{\CP{G}{K_n}}\leq
  n(\tw{G}+1)-1$. Every graph $H$ satisfies $\eta(H)\leq\tw{H}+1$. The
  result follows.

  Now we prove the lower bound. Let $G$ be the graph with vertex
  set $$V(G):=\{v_1,\dots,v_{k+1}\}\cup\{x_{i,j,p}:i,j\in[k+1],p\in[n]\}\enspace,$$
  where $\{v_1,\dots,v_{k+1}\}$ is a clique, and each $x_{i,j,p}$ is
  adjacent to $v_i$ and $v_j$. A tree decomposition $T$ of $G$ is
  constructed as follows. Let $T_r:=\{v_1,\dots,v_{k+1}\}$, and for
  all $i,j\in[k+1]$ and $p\in[n]$, let
  $T_{i,j,p}:=\{x_{i,j,p},v_i,v_j\}$, where $T_r$ is adjacent to every
  $T_{i,j,p}$ and there are no other edges in $T$. Thus $T$ is a star
  with $n(k+1)^2$ leaves, and $(T,\{T_x:x\in V(T)\})$ is a tree
  decomposition of $G$ with at most $k+1$ vertices in each bag. Thus
  $\tw{G}\leq k$. Since $G$ contains a clique of $k+1$ vertices,
  $\tw{G}=k$.

  Now consider \CP{G}{K_n}. For $i,j\in[k+1]$ and $p\in[n]$, let
  $A\langle i,j,p\rangle$ be the subgraph of \CP{G}{K_n} induced by
  $\{(x_{i,j,p},q):q\in[n]\}$. Thus $A\langle{i,j,p}\rangle$ is a copy
  of $K_n$.  For $i\in[k+1]$ and $p\in[n]$, let $X\langle{i,p}\rangle$
  be the subgraph induced by $\cup\{A\langle i,j,p\rangle:j\in[k+1]\}$
  plus the vertex $(v_i,p)$.  We claim that the
  $X\langle{i,p}\rangle$ are the branch set of clique minor in
  \CP{G}{K_n}. First we prove that each $X\langle{i,p}\rangle$ is
  connected.  For all $j\in[k+1]$, $v_i$ is adjacent to $x_{i,j,p}$ in
  $G$. Thus $(v_i,p)$ is adjacent to $(x_{i,j,p},p)$, which is in
  $A\langle{i,j,p}\rangle\subset X\langle{i,p}\rangle$. Thus $X\langle
  i,p\rangle$ consists of $k+1$ copies of $K_n$ plus one vertex
  adjacent to each copy. In particular, $X\langle i,p\rangle$ is
  connected. Now consider distinct subgraphs $X\langle{i,p}\rangle$
  and $X\langle{j,q}\rangle$. The first coordinate of every vertex in
  $X\langle{i,p}\rangle$ is either $v_{i,p}$ or $x_{i,i',p}$ for some
  $i'\in[k+1]$. Thus $X\langle{i,p}\rangle$ and $X\langle{j,q}\rangle$
  are disjoint.  Now the vertex $x_{i,j,p}$ is adjacent to the vertex
  $v_{j,q}$ in $G$. Thus the vertex $(x_{i,j,p},q)$, which is in
  $A\langle i,j,p\rangle\subset X\langle i,p\rangle$, is adjacent to
  the vertex $(v_{j,q},q)$, which is in $X\langle j,q\rangle$. Thus
  $X\langle{i,p}\rangle$ and $X\langle{j,q}\rangle$ are
  adjacent. Hence the $X\langle{i,p}\rangle$ are the branch set of
  clique minor in \CP{G}{K_n}, and $\eta(\CP{G}{K_n})\geq n(k+1)$. We
  have equality because of the above upper bound.
\end{proof}

We have the following similar upper bound for the bandwidth of a
product.

\begin{lemma}
  \lemlabel{BandwidthProduct} For every graph $G$ and integer
  $n\geq1$,
 $$\bw{\CP{G}{K_n}}\leq n\cdot\bw{G}\enspace.$$
\end{lemma}

\begin{proof}
  Say $V(K_n)=\{w_1,\dots,w_n\}$.  Let $(v_1,\dots,v_p)$ be a vertex
  ordering of $G$, such that $\max\{|i-j|:v_iv_j\in
  E(G)\}=\bw{G}$. Order the vertices of \CP{G}{K_n},
$$(v_1,w_1),\dots,(v_1,w_n);
(v_2,w_1),\dots,(v_2,w_n);\,\ldots;\,
(v_p,w_1),\dots,(v_p,w_n)\enspace.$$ In this ordering, an edge
$(v_i,w_j)(v_i,w_{\ell})$ of \CP{G}{K_n} has length at most $n-1$, and
an edge $(v_i,w_\ell)(v_j,w_\ell)$ of \CP{G}{K_n} has length
$n\cdot|i-j|\leq n\cdot\bw{G}$ (since $v_iv_j\in E(G)$). Thus
$\bw{\CP{G}{K_n}}\leq n\cdot\bw{G}$ (since $n\cdot\bw{G}\geq n-1$).
\end{proof}

We now set out to prove a lower bound on $\eta(\CP{G}{K_n})$ in terms
of the treewidth of $G$. We start by considering the case $n=2$, which
is of particular importance in \secref{HadwigerConjecture}
below. \citet{RS-GraphMinorsV-JCTB86} proved that every graph with
large treewidth has a large grid minor. The following explicit bound
was obtained by \citet{DJGT-JCTB99}; also see \citep[Theorem
12.4.4]{Diestel00}.

\begin{lemma}[\citep{DJGT-JCTB99}]
  \lemlabel{BigGrid} For all integers $k,m\geq1$ every graph with
  tree-width at least $k^{4m^2(k+2)}$ contains $P_k\square P_k$ or
  $K_m$ as a minor. In particular, every graph with tree-width at
  least $k^{4k^4(k+2)}$ contains a $P_k\square P_k$-minor.
\end{lemma}

In what follows all logarithms are binary.

\begin{lemma}
  \lemlabel{LowerBoundGraphEdge} For every graph $G$ with at least one
  edge,
$$\eta(\CP{G}{K_2})>(\tfrac14\log\tw{G})^{1/4}\enspace.$$
\end{lemma}

\begin{proof}
  Let $\ell$ be the real-valued solution to
  $\tw{G}=\ell^{4(\ell+1)^3}$. Thus $\ell\geq1$, and
$$\log\tw{G}=4(\ell+1)^3(\log\ell)< 4(\ell+1)^4\enspace.$$
That is, $(\tfrac{1}{4}\log\tw{G})^{1/4}<\ell+1$.  Let
$k:=\floor{\ell}$. Thus $k\geq1$ and $\tw{G}\geq k^{4(k+1)^3}$.  Hence
\lemref{BigGrid} is applicable with $m=k+1$.  Thus $G$ contains
$P_k\square P_k$ or $K_{k+1}$ as a minor.  If $G$ contains a
$P_k\square P_k$-minor, then \CP{G}{K_2} contains a $K_{k+2}$-minor by
\eqnref{DoubleGrid}.  Otherwise $G$ contains a $K_{k+1}$ minor, and by
\propref{CompleteEdge}, \CP{G}{K_2} contains a $K_{k+2}$-minor.  In
both cases
$$\eta(\CP{G}{K_2})
\geq k+2 >\ell+1 >(\tfrac{1}{4}\log\tw{G})^{1/4}\enspace,$$ as
desired.
\end{proof}

\lemref{LowerBoundGraphEdge} and \thmref{TreewidthComplete} imply that
$\eta(\CP{G}{K_2})$ is tied to the treewidth of $G$. In particular,
\begin{equation}
  \eqnlabel{TreewdithTied}
  (\tfrac14\log\tw{G})^{1/4}< \eta(\CP{G}{K_2})\leq 2\,\tw{G}+2\enspace.
\end{equation}
This result is similar to a theorem by \citet{BM69}, who proved that
\CP{G}{K_2} is planar if and only if $G$ is outerplanar.
\Eqnref{TreewdithTied} says that \CP{G}{K_2} has bounded $\eta$ if and
only if $G$ has bounded treewidth.

We now extend \lemref{LowerBoundGraphEdge} for general complete
graphs.


\begin{lemma}
  \lemlabel{LowerBoundGraphComplete} For every graph $G$ with at least
  one edge and every integer $n\geq1$,
$$\eta(\CP{G}{K_n})>\floor{\tfrac{n}{2}}\big(\tfrac{1}{16}\log\tw{G}\big)^{1/6}\enspace.$$
\end{lemma}

\begin{proof}
  Let $\ell$ be the real-valued solution to
  $\tw{G}=\ell^{4\ell^4(\ell+2)}$. Thus $\ell\geq 1$.
  Thus $$\log\tw{G}=4\ell^4(\ell+2)(\log \ell)\leq 12\ell^6\enspace.$$
  That is, $(\tfrac{1}{16}\log\tw{G})^{1/6}\leq \ell$.  Let
  $k:=\floor{\ell}$.  Thus $\ell\geq k\geq1$ and $\tw{G}\geq
  k^{4k^4(k+2)}$.  By \lemref{BigGrid}, $G$ contains a $P_k\square
  P_k$-minor.  By \lemref{PathPathComplete} below,
$$\eta(\CP{G}{K_n})
\geq\floor{\tfrac{n}{2}}(k+1) >\floor{\tfrac{n}{2}}\ell
\geq\floor{\tfrac{n}{2}}\big(\tfrac{1}{16}\log\tw{G}\big)^{1/6}\enspace,$$
as desired.
\end{proof}

\lemref{LowerBoundGraphComplete} and \thmref{TreewidthComplete} imply
that $\eta(\CP{G}{K_n})/n$ is tied to the treewidth of $G$. In
particular,
$$\big(\tfrac{1}{16}\log\tw{G}\big)^{1/6} 
< \frac{\eta(\CP{G}{K_n})}{n} \leq \tw{G}+1\enspace.$$

It remains to prove \lemref{PathPathComplete}.

\begin{lemma}
  \lemlabel{PathPathComplete} For all integers $n\geq1$ and $k\geq1$
$$(k+1)\floor{\tfrac{n}{2}}\leq 
\eta(\TCP{P_k}{P_k}{K_n}) <k(n+\half)+3\enspace.$$
\end{lemma}

\begin{proof}
  Since $\TCP{P_k}{P_k}{K_n}$ has $k^2n$ vertices and maximum degree
  $n+3$, \lemref{UpperBound} implies the upper bound,
  $$\eta(\TCP{P_k}{P_k}{K_n})\leq\sqrt{(n+1)k^2n}+3<k(n+\half)+3\enspace.$$

  Now we prove the lower bound. Let $p:=\floor{\frac{n}{2}}$.  Each
  vertex is described by a triple $(x,y,r)$ where $x,y\in[k]$ and
  $r\in[n]$. Distinct vertices $(x,y,r)$ and $(x',y',r')$ are adjacent
  if and only if $x=x'$ and $y=y'$, or $x=x'$ and $|y-y'|=1$ and
  $r=r'$, or $y=y'$ and $|x-x'|=1$ and $r=r'$.

  For $r\in[p]$, let $T_{0,r}$ be the subgraph induced by
  $\{(1,y,2r-1):y\in[k]\}$, and let $T_{1,r}$ be the subgraph induced
  by $\{(x,1,2r):x\in[k]\}$. For $i\in[2,k]$ and $r\in[p]$, let
  $T_{i,r}$ be the subgraph of $\TCP{P_k}{P_k}{K_n}$ induced
  by $$\{(i,y,2r-1):y\in[k]\}\cup\{(x,i,2r):x\in[k]\}\enspace.$$

  For all $r\in[p]$, both $T_{0,r}$ and $T_{1,r}$ are paths, and for
  $i\in[2,k]$, $T_{i,r}$ consists of two adjacent paths. In
  particular, each $T_{i,r}$ is connected. Observe that each pair of
  distinct subgraphs $T_{i,r}$ and $T_{j,s}$ are disjoint.

  There is an edge from $(1,1,2r-1)$ in $T_{0,r}$ to $(1,1,2s)$ in
  $T_{1,s}$.  For all $i\in[2,k]$, there is an edge from $(1,i,2r)$ in
  $T_{i,r}$ to $(1,i,2s-1)$ in $T_{0,s}$, there is an edge from
  $(i,1,2r-1)$ in $T_{i,r}$ to $(i,1,2s)$ in $T_{1,s}$, and for all
  $i,j\in[2,k]$, there is an edge from $(j,i,2r)$ in $T_{i,r}$ to
  $(j,i,2s-1)$ in $T_{j,s}$.

  Hence the $T_{i,j}$ are the branch sets of a
  $K_{(k+1)m}$-minor. Therefore
  \parbox{\textwidth-10mm}{$$\eta(\TCP{P_k}{P_k}{K_n})\geq(k+1)p=(k+1)\FLOOR{\tfrac{n}{2}}\enspace.$$}
\end{proof}

\section{Rough Structural Characterisation Theorem} \seclabel{RoughGraphs}

In this section we give a rough structural characterisation of pairs
of graphs whose product has bounded Hadwiger number. The proof is
based heavily on the corresponding result for trees in
\secref{TwoTrees}. Thus our first task is to extend a number of
definitions for trees to general graphs.

For a connected graph $G$, let \bal{G} be the maximum order of a
balance subgraph in $G$. A path $P$ in $G$ is \emph{semi-clean} if
every internal vertex of $P$ has degree $2$ in $G$. Let $p'(G)$ be the
maximum number of vertices in a semi-clean path in $G$. A path $P$ is
\emph{clean} if it is semi-clean, and every edge of $P$ is a cut in
$G$. Let $p(G)$ be the maximum number of vertices in a clean path in
$G$. Note that $p(G)\geq1$ since a single vertex is a clean path. In
fact, if $G$ is $2$-connected, then the only clean paths are single
vertices, and $p(G)=1$. On the other hand, since every edge in a tree
$T$ is a cut, our two definitions of a clean path are equivalent for
trees, and $p'(T)=p(T)$.

The \emph{hangover} of a connected graph $G$, denoted by \hang{G}, is
defined as follows. If $G$ is a path or a cycle then
$\hang{G}:=0$. Otherwise, \hang{G} is the minimum, taken over all
clean paths $P$ in $G$, of the maximum number of vertices in a
component of $G-E(P)$. First note the following trivial relationship
between \hang{G} and $p(G)$:
\begin{equation}
  \eqnlabel{HangTrivial}
  \half(\verts{G}-p(G)+2)\leq \hang{G}\leq\verts{G}-p(G)+1.
\end{equation}
To prove a relationship between \bal{G} and \hang{G} below, we reduce
the proof to the case of trees using the following lemma.

\begin{lemma}
  \lemlabel{GoodTree} Every connected graph $G$ has a spanning tree
  $T$ such that $$p(T)\leq p'(G)+6\enspace.$$
\end{lemma}

\begin{proof}
  Define a \emph{leaf-neighbour} in a tree to be a vertex of degree
  $2$ that is adjacent to a leaf (a vertex of degree $1$).

  Choose a spanning tree $T$ of $G$ that firstly maximises the number
  of leaves in $T$, and secondly maximises the number of
  leaf-neighbours in $T$.

  If $p(T)\leq 7$ then the claim is vacuous since $p'(G)\geq1$. Now
  assume that $p(T)\geq8$. Let $(v_1,\dots,v_k)$ be a clean path in
  $T$ with $k=p(T)$. Below we prove that $\deg_G(v_i)=2$ for each
  $i\in[4,k-3]$. This shows that the path $(v_4,\dots,v_{k-3})$ is
  semi-clean in $G$, implying $p'(G)\geq p(T)-6$, as desired.

  Suppose on the contrary that $\deg_G(v_i)\geq 3$ for some
  $i\in[4,k-3]$.  Let $w$ be a neighbour of $v_i$ in $G$ besides
  $v_{i-1}$ and $v_{i+1}$. Without loss of generality, the path
  between $v_i$ and $w$ in $T$ includes $v_{i+1}$.

  \textbf{Case 1.} $\deg_T(w)\geq2$: Let $T'$ be the spanning tree of
  $G$ obtained from $T$ by deleting the edge $v_iv_{i+1}$ and adding
  the edge $v_iw$.  Now $\deg_T(v_{i+1})=2$ (since $i+1\leq
  k-1$). Thus $v_{i+1}$ becomes a leaf in $T'$.  Since
  $\deg_{T}(v_i)=\deg_{T'}(v_i)=2$, $v_i$ is a leaf in neither $T$ nor
  $T'$.  Since $\deg_T(w)\geq2$ and $\deg_{T'}(w)\geq3$, $w$ is a leaf
  in neither $T$ nor $T'$.  The degree of every other vertex is
  unchanged.  Hence $T'$ has one more leaf than $T$.  This contradicts
  the choice of $T$.

  Now assume that $\deg_T(w)=1$. Let $x$ be the neighbour of $w$ in
  $T$.

  \textbf{Case 2.} $\deg_T(w)=1$ and $\deg_T(x)=2$: Let $T'$ be the
  spanning tree of $G$ obtained from $T$ by deleting the edge $wx$ and
  adding the edge $v_iw$.  Since $\deg_{T}(v_i)=2$ and
  $\deg_{T'}(v_i)=3$, $v_i$ is a leaf in neither $T$ nor $T'$.  Since
  $\deg_T(w)=\deg_{T'}(w)=1$, $w$ is a leaf in both $T$ and $T'$.
  Since $\deg_T(x)=2$ and $\deg_{T'}(x)=1$, $x$ becomes a leaf $T'$.
  The degree of every other vertex is unchanged.  Hence $T'$ has one
  more leaf than $T$.  This contradicts the choice of $T$.

  \textbf{Case 3.} $\deg_T(w)=1$ and $\deg_T(x)\geq 3$: Let $T'$ be
  the spanning tree of $G$ obtained from $T$ by deleting the edge
  $v_iv_{i+1}$ and adding the edge $v_iw$.  Since
  $\deg_{T}(v_i)=\deg_{T'}(v_i)=2$, $v_i$ is a leaf in neither $T$ nor
  $T'$.  Now $\deg_T(v_{i+1})=2$ (since $i+1\leq k-1$).  Thus
  $v_{i+1}$ is a leaf in $T'$ but not in $T$.  Since $\deg_T(w)=1$ and
  $\deg_{T'}(w)=2$, $w$ is a leaf in $T$ but not in $T'$.  The degree
  of every other vertex is unchanged.  Hence $T'$ has the same number
  of leaves as $T$.

  Suppose, for the sake of contradiction, that there is a
  leaf-neighbour $p$ in $T$ that is not a leaf-neighbour in $T'$.
  Since $v_{i+1}$ and $w$ are the only vertices with different degrees
  in $T$ and $T'$, $p$ is either $v_{i+1}$ or $w$, or $p$ is a
  neighbour of $v_{i+1}$ or $w$ in $T$ or $T'$.  That is,
  $p\in\{v_{i+1},w,x,v_i,v_{i+2}\}$.  Now $\deg_T(p)=2$ since $p$ is a
  leaf-neighbour in $T$.  Thus $p\neq w$ and $p\neq x$.  Every
  neighbour of $v_i$ and $v_{i+1}$ in $T$ has degree $2$ in $T$ (since
  $i-1\geq 2$ and $i+2\leq k-1$).  Thus $p\neq v_i$ and $p\neq
  v_{i+1}$.  Finally, $p\neq v_{i+2}$ since the neighbours of
  $v_{i+2}$ in $T$, namely $v_{i+1}$ and $v_{i+3}$, are both not
  leaves (since there is a path in $T$ from $v_{i+3}$ to $x$ that
  avoids $v_{i+2}$).  This contradiction proves that every
  leaf-neighbour in $T$ is also a leaf-neighbour in $T'$.

  Now consider the vertex $v_{i+2}$. In both $T$ and $T'$, the only
  neighbours of $v_{i+2}$ are $v_{i+1}$ and $v_{i+3}$ (since $i+2\leq
  k-1$). Both $v_{i+1}$ and $v_{i+3}$ have degree $2$ in $T$, but
  $v_{i+1}$ is a leaf in $T'$. Thus $v_{i+2}$ is a leaf-neighbour in
  $T'$, but not in $T$.

  Hence $T'$ has more leaf-neighbours than $T$. This contradicts the
  choice of $T$, and completes the proof.
\end{proof}

\lemref{BalanceHang} proves that \bal{T} and \hang{T} are tied for
trees.  We now prove an analogous result for general graphs.

\begin{lemma}
  \lemlabel{BalanceHangGraph} For every connected graph $G$,
$$\hang{G}\leq 8\,\bal{G}+9\enspace.$$
\end{lemma}

\begin{proof}
  If $G$ is a path or cycle, then $\hang{G}=0$ and the result is
  vacuous.  Now assume that $G$ is neither a path nor a cycle.  By
  \lemref{GoodTree}, $G$ has a spanning tree $T$ such that $p(T)\leq
  p'(G)+6$. By \lemref{BalanceHang},
$$\bal{G}\geq\bal{T}\geq\third(\hang{T}-1)\enspace.$$
By the lower bound in \eqnref{HangTrivial},
\begin{align*}
  \bal{G}\geq\third(\half(\verts{T}-p(T)+2)-1)
  =\;&\sixth(\verts{G}-p(T))\\
  \geq\;& \sixth(\verts{G}-p'(G)-6)\enspace.
\end{align*}
Thus we are done if
$\sixth(\verts{G}-p'(G)-6)\geq\eighth(\hang{G}-9)$.  Now assume that
$$\sixth(\verts{G}-p'(G)-6)\leq\eighth(\hang{G}-9)\enspace.$$
By the upper bound in \eqnref{HangTrivial},
$$\sixth(\verts{G}-p'(G)-6)\leq\eighth(\verts{G}-p(G)+1-9)\enspace.$$
That is,
\begin{align}
  \eqnlabel{SomeSomeEqn} \verts{G}+3p(G)\leq4p'(G)\enspace.
\end{align}
If $p(G)\geq p'(G)$, then $\verts{G}\leq p'(G)$, which implies that
$G$ is a path.  Now assume that $p(G)\leq p'(G)-1$. Thus there is a
non-clean semi-clean path $P$ in $G$ of length $p'(G)$. Since $P$ is
not clean and $G$ is connected and not a cycle, there is a cycle $C$
in $G$ with at least $p'(G)$ vertices, such that one vertex $r$ in $C$
is adjacent to a vertex $s$ not in $C$. It follows that $G$ has a
balance rooted at $r$ with support $s$, and with order at least
$\floor{\half(p'(G)-1)}$. Thus $\bal{G}\geq\half p'(G)-1$. That is,
$8\,\bal{G}+8\geq 4p'(G)$. By \eqnref{SomeSomeEqn},
$$8\,\bal{G}+8\geq \verts{G}+3p(G)
\geq \verts{G} \geq \hang{G}\enspace,$$ as desired.
\end{proof}


We now prove an analogue of \lemref{BigTrees} for general graphs.

\begin{lemma}
  \lemlabel{BigGraphs} Fix an integer $c\geq1$. Let $G$ and $H$ be
  graphs, such that $\verts{G}\geq2c^2-c+2$, $\verts{H}\geq c+1$, and
  $\eta(\CP{G}{H})\leq c$. Then $$\hang{H}\leq 8c+9\enspace.$$ By
  symmetry, if in addition $\verts{H}\geq 2c^2-c+2$ then
$$\hang{G}\leq 8c+9\enspace.$$
\end{lemma}

\begin{proof}
  If $\STAR{G}\geq c$, then by \corref{GraphStarMinor},
  $$\eta(\CP{G}{H})\geq \min\{\verts{H},\STAR{G}+1\}\geq c+1\enspace,$$ which
  contradicts the assumption. Now assume that $\STAR{G}\leq c-1$.

  Let $T$ be a spanning tree of $G$.  Let $n:=c+1$. Then
  \begin{align*}
    \verts{T} =\verts{G} \geq (c+1)^2+(c-3)c+1
    &\geq n^2+(\STAR{G}-2)(n-1)+1\\
    &\geq n^2+(\STAR{T}-2)(n-1)+1\enspace.
  \end{align*}
  Thus \thmref{TreePartition} is applicable to $T$ with $n=c+1$.
  Hence $T$ has $c+1$ disjoint subtrees, each with at least $c+1$
  vertices.

  If $\bal{H}\geq c+1$, then by \thmref{BigBalance},
  $$\eta(\CP{G}{H})\geq\min\{c+1,\bal{H}\}=c+1\enspace,$$ which contradicts the
  assumption.  Thus $\bal{H}\leq c$.  Hence, by
  \lemref{BalanceHangGraph}, $\hang{H}\leq 8c+9$.
\end{proof}


We now prove that the product of graphs with bounded hangover have a
specific structure.

\begin{lemma}
  \lemlabel{BoundedHang} Fix an integer $c\geq1$. For all graphs $G$
  and $H$, if $\hang{G}\leq c$ and $\hang{H}\leq c$, then \CP{G}{H} is
  one of the following graphs:
  \begin{itemize}
  \item a planar grid (the product of two paths) with a vortex of
    width at most $2c^2$ in the outerface,
  \item a cylindrical grid (the product of a path and a cycle) with a
    vortex of width at most $2c$ in each of the two `big' faces, or
  \item a toroidal grid (the product of two cycles).
  \end{itemize}
\end{lemma}

\begin{proof}
  If $G$ and $H$ are cycles then \CP{G}{H} is a toroidal grid.  If
  neither $G$ nor $H$ are cycles then by the same argument used in the
  proof of \lemref{RoughTree}, \CP{G}{H} is obtained from a planar
  $p(G)\times p(H)$ grid by adding a vortex in the outerface with
  width at most $2c^2$.  If $G$ is a cycle and $H$ is not a cycle,
  then by a similar argument used in the proof of \lemref{RoughTree},
  \CP{G}{H} is obtained from a cylindrical $\verts{C_n}\times p(H)$
  grid by adding a vortex in each of the two `big' faces with width at
  most $2c$.
\end{proof}


\twolemref{BigGraphs}{BoundedHang} imply the following
characterisation of large graphs with bounded Hadwiger number that was
described in \secref{Intro}.

\begin{theorem}
  \thmlabel{SomeTheorem} Fix an integer $c\geq1$. For all graphs $G$
  and $H$ with $\verts{G}\geq 2c^2-c+2$ and $\verts{H}\geq 2c^2-c+2$,
  if $\eta(\CP{G}{H})\leq c$ then \CP{G}{H} is one of the following
  graphs:
  \begin{itemize}
  \item a planar grid (the product of two paths) with a vortex of
    width at most $2(8c+9)^2$ in the outerface,
  \item a cylindrical grid (the product of a path and a cycle) with a
    vortex of width at most $16c+18$ in each of the two `big' faces,
    or
  \item a toroidal grid (the product of two cycles).
  \end{itemize}
\end{theorem}

\thmref{SomeTheorem} is similar to a result by \citet{BM69}, who
proved that if $G$ and $H$ are connected graphs with at least $3$
vertices, then \CP{G}{H} is planar if and only if both $G$ and $H$ are
paths, or one is a path and the other is a cycle.


We now prove the first part of our rough structural characterisation
of graph products with bounded Hadwiger number.

\begin{lemma}
  \lemlabel{BoundedHadwigerImplies} Let $G$ and $H$ be connected
  graphs, each with at least one edge, such that $\eta(\CP{G}{H})\leq
  c$ for some integer $c$. Then for some integers $c_1,c_2,c_3$
  depending only on $c$:
  \begin{itemize}
  \item $\tw{G}\leq c_1$ and $\verts{H}\leq c_2$, or
  \item $\tw{H}\leq c_1$ and $\verts{G}\leq c_2$, or
  \item $\hang{G}\leq c_3$ and $\hang{H}\leq c_3$.
  \end{itemize}
\end{lemma}

\begin{proof}
  Let $c_1:=2^{4c^4}$, $c_2:=2c^2-c+1$, and $c_3:=8c+9$.

  First suppose that $\tw{G}>c_1$ or $\tw{H}>c_1$.  Without loss of
  generality, $\tw{G}>c_1$.  Then by \lemref{LowerBoundGraphEdge},
  $\eta(\CP{G}{H})\geq\eta(\CP{G}{K_2})>(\tfrac14\log c_1)^{1/4}=c,$
  which is a contradiction.  Now assume that $\tw{G}\leq c_1$ and
  $\tw{H}\leq c_1$.

  Thus, if $\verts{H}\leq c_2$ or $\verts{G}\leq c_2$, then the first
  or second condition is satisfied, and we are done.  Now assume that
  $\verts{H}>c_2$ and $\verts{G}> c_2$.  By \lemref{BigGraphs},
  \hang{G} and \hang{H} are both at most $8c+9=c_3$, as desired.
\end{proof}


Now we prove the converse of \lemref{BoundedHadwigerImplies}.

\begin{lemma}
  \lemlabel{ImpliesBoundedHadwiger} Let $G$ and $H$ be connected
  graphs, each with at least one edge, such that for some integers
  $c_1,c_2,c_3$,
  \begin{itemize}
  \item $\tw{G}\leq c_1$ and $\verts{H}\leq c_2$, or
  \item $\tw{H}\leq c_1$ and $\verts{G}\leq c_2$, or
  \item $\hang{G}\leq c_3$ and $\hang{H}\leq c_3$.
  \end{itemize}
  Then $\eta(\CP{G}{H})\leq c$ where $c$ depends only on
  $c_1,c_2,c_3$.
\end{lemma}

\begin{proof}
  Suppose that $\tw{G}\leq c_1$ and $\verts{H}\leq c_2$.  
  \thmref{TreewidthComplete} implies
  that $$\eta(\CP{G}{H})\leq\eta(\CP{G}{K_{c_2}})\leq
  c_2(\tw{G}+1)\leq c_2(c_1+1)\enspace,$$ and we are done.  Similarly,
  if $\tw{H}\leq c_1$ and $\verts{G}\leq c_2$, then
  $\eta(\CP{G}{H})\leq c_2(c_1+1)$, and we are done.  Otherwise,
  $\hang{G}\leq c_3$ and $\hang{H}\leq c_3$.  By \lemref{BoundedHang},
  \CP{G}{H} is either a toroidal grid (which has no $K_8$ minor), or
  \CP{G}{H} is a planar graph plus vortices of width at most $2c^2$ in
  one or two of the faces. As in the proof of   \lemref{RoughTree}, it follows that $\eta(\CP{G}{H})\leq\Oh{c^2}$.
\end{proof}


\twolemref{BoundedHadwigerImplies}{ImpliesBoundedHadwiger} imply the
following rough structural characterisation of graph products with
bounded Hadwiger number.

\begin{theorem}
  \thmlabel{RoughGraph} The function $\eta(\CP{G}{H})$ is tied to
$$\min\big\{
\max\{\tw{G},\verts{H}\},\, \max\{\verts{G},\tw{H}\},\,
\max\{\hang{G},\hang{H}\} \,\big\} \enspace.$$
\end{theorem}

\thmref{RoughGraph} can be informally stated as: $\eta(\CP{G}{H})$ is
bounded if and only if:
\begin{itemize}
\item $\tw{G}$ and $\verts{H}$ is bounded, or
\item $\verts{G}$ and $\tw{H}$ is bounded, or
\item $\hang{G}$ and $\hang{H}$ are bounded.
\end{itemize}

\section{On Hadwiger's Conjecture for Cartesian Products}
\seclabel{HadwigerConjecture}

In 1943, \citet{Hadwiger43} made the following conjecture:

\medskip\noindent\textbf{Hadwiger's Conjecture}. For every graph
$G$, $$\chi(G)\leq\eta(G)\enspace.$$

This conjecture is widely considered to be one of the most significant
open problems in graph theory; see the survey by
\citet{Toft-HadwigerSurvey96}. Yet it is unknown whether Hadwiger's
conjecture holds for all non-trivial products. (We say \CP{G}{H} is
\emph{non-trivial} if both $G$ and $H$ are both connected and have at
least one edge.)\ The chromatic number of a product is well
understood. In particular, \citet{Sabidussi57} proved that
$\chi(\CP{G}{H})=\max\{\chi(G),\chi(H)\}$. Thus Hadwiger's Conjecture
for products asserts
that $$\max\{\chi(G),\chi(H)\}\leq\eta(\CP{G}{H})\enspace.$$

Hadwiger's Conjecture is known to hold for various classes of
products. For example, \citet{CS-DM07} proved that the product of
sufficiently many graphs (relative to their maximum chromatic number)
satisfies Hadwiger's Conjecture.
The best bounds are by \CR, who proved that for some constant $c$,
Hadwiger's Conjecture holds for the non-trivial product
\CCP{G_1}{G_2}{G_d} whenever
$$\max_i\chi(G_i)\leq 2^{2^{(d-c)/2}}\enspace.$$ 

In a different direction, \CR\ proved that if $\chi(G)\geq\chi(H)$ and
$\chi(H)$ is not too small relative to $\chi(G)$, then \CP{G}{H}
satisfies Hadwiger's Conjecture. In particular, there is a constant
$c$, such that if $\chi(G)\geq\chi(H)\geq c\log^{3/2}\chi(G)$ then
\CP{G}{H} satisfies Hadwiger's Conjecture. Similarly, they also
implicitly proved
that $$\min\{\chi(G),\chi(H)\}\leq\eta(\CP{G}{H})\enspace,$$ and
concluded that if $\chi(G)=\chi(H)$ then Hadwiger's Conjecture holds
for \CP{G}{H}. We make the following small improvement to this result.

\begin{lemma}
  \lemlabel{MinChi} For all connected graphs $G$ and $H$, both with at
  least one edge,
$$\min\{\chi(G),\chi(H)\}\leq\eta(\CP{G}{H})-1\enspace.$$
Moreover, if $G\neq K_2$ and $H\neq K_2$ then
$$\min\{\chi(G),\chi(H)\}\leq\eta(\CP{G}{H})-2\enspace.$$
\end{lemma}

\begin{proof}
  We have $\chi(G)\leq\Delta(G)+1$ and $\chi(H)\leq\Delta(H)+1$. Thus
  by \corref{GraphStarMinor},
$$\min\{\chi(G),\chi(H)\}\leq\min\{\Delta(G),\Delta(H)\}+1\leq
\eta(\CP{G}{H})-1\enspace.$$ Now assume that $G\neq K_2$ and $H\neq
K_2$.

\textbf{Case 1. } $G\in\{C_n,K_n\}$ and $\eta(H)=2$ for some $n\geq3$:
Then $H$ is a tree and $\min\{\chi(G),\chi(H)\}=2$. On the other hand,
$\eta(\CP{G}{H})\geq\eta(\CP{K_3}{K_2})=4$ by \propref{CompleteEdge}.

\textbf{Case 2. } $G=K_n$ and $\eta(H)\geq 3$ for some $n\geq3$: Then
$\min\{\chi(G),\chi(H)\}\leq n$ and
$\eta(\CP{G}{H})\geq\eta(\CP{K_n}{K_3})=n+2$ by
\propref{TriangleComplete}.

\textbf{Case 3. } $G=C_n$ and $\eta(H)\geq3$ for some $n\geq3$: Then
$\min\{\chi(G),\chi(H)\}\leq 3$, and
$\eta(\CP{G}{H})\geq\eta(\CP{C_3}{C_3})=5$ by a result of
\citet{ABPS97}. (In fact, \citet{ABPS97} determined
$\eta(\CP{C_n}{C_m})$ for all values of $n$ and $m$, as described in
\tabref{Cycles}. \citet{Miller-DM78} had previously stated without
proof that $\eta(\CP{C_n}{K_2})=4$ for all $n\geq3$.)\

\textbf{Case 4. } Both $G$ and $H$ are neither complete graphs nor
cycles. Then by Brooks' Theorem \citep{Brooks41},
$\chi(G)\leq\Delta(G)$ and $\chi(H)\leq\Delta(H)$. Thus
$\min\{\chi(G),\chi(H)\}\leq\min\{\Delta(G),\Delta(H)\}$.  By
\corref{GraphStarMinor},
$\eta(\CP{G}{H})\geq\min\{\Delta(G),\Delta(H)\}+2$.  Thus
$\min\{\chi(G),\chi(H)\}\leq\eta(\CP{G}{H})-2$.
\end{proof}

\begin{table}[!htb]
  \caption{The Hadwiger number of \CP{C_n}{C_m}, where $C_2=K_2$; see \citep{ABPS97}.}
  \tablabel{Cycles}
  \vspace*{1ex}
  \begin{tabular}{c|ccccc}
    \hline
    & $n=2$		& $n=3$		& $n=4$		& $n=5$		& $n\geq 6$	\\\hline
    $m=2$		& $3$		& $4$		& $4$		& $4$		& $4$		\\
    $m=3$		& $4$		& $5$		& $5$		& $5$		& $6$		\\
    $m=4$		& $4$		& $5$		& $6$		& $6$		& $7$ 		\\
    $m=5$		& $4$		& $5$		& $6$		& $7$		& $7$ 		\\
    $m\geq6$	& $4$		& $6$		& $7$		& $7$		& $7$ 		\\\hline
  \end{tabular}
\end{table}

\begin{theorem}
  \thmlabel{HadwigerNearChi} Hadwiger's Conjecture holds for a
  non-trivial product \CP{G}{H} whenever $|\chi(G)-\chi(H)|\leq
  2$. Moreover, if $|\chi(G)-\chi(H)|\leq 1$ then
  $\chi(\CP{G}{H})\leq\eta(\CP{G}{H})-1$.
\end{theorem}

\begin{proof}
  Without loss of generality, $\chi(G)-2\leq\chi(H)\leq\chi(G)$.
  Thus, by Sabidussi's Theorem \citep{Sabidussi57}, it suffices to
  prove that $\chi(G)\leq\eta(\CP{G}{H})$.  If $H=K_2$ then
  $\chi(G)\leq 4$ by assumption.  \citet{Hadwiger43} and
  \citet{Dirac52} independently proved Hadwiger's Conjecture whenever
  $\chi(G)\leq 4$.  Thus $\eta(\CP{G}{H})\geq\eta(G)+1\geq\chi(G)+1$,
  as desired.  This proves the `moreover' claim in this case.  Now
  assume that $H\neq K_2$. Thus by \lemref{MinChi},
  $\chi(G)\leq\chi(H)+2\leq\eta(\CP{G}{H})$, as desired.  And if
  $\chi(G)\leq \chi(H)+1$ then $\chi(G)\leq\eta(\CP{G}{H})-1$, as
  desired.
\end{proof}


The following two theorems establish Hadwiger's Conjecture for new
classes of products. The first says that products satisfy Hadwiger's
Conjecture whenever one graph has large treewidth relative to its
chromatic number.

\begin{theorem}
  \thmlabel{HadwigerBigTW} Hadwiger's Conjecture is satisfied for the
  non-trivial product \CP{G}{H} whenever $\chi(G)\geq\chi(H)$ and $G$
  has treewidth $\tw{G}\geq2^{4\chi(G)^4}$.
\end{theorem}

\begin{proof}
  Since $H$ has at least one edge,
  $\eta(\CP{G}{H})\geq\eta(\CP{G}{K_2})$.  By
  \lemref{LowerBoundGraphEdge},
  $\eta(\CP{G}{H})>(\tfrac14\log\tw{G})^{1/4}$, which is at least
  $\chi(G)$ by assumption.  Hence
  $\eta(\CP{G}{H})>\chi(G)=\chi(\CP{G}{H})$ by Sabidussi's Theorem
  \citep{Sabidussi57}.  That is, \CP{G}{H} satisfies Hadwiger's
  Conjecture.
\end{proof}

We now show that products satisfy (a slightly better bound than)
Hadwiger's Conjecture whenever the graph with smaller chromatic number
is relatively large.

\begin{theorem}
  \thmlabel{HadwigerBigH} Let $G$ and $H$ be connected graphs with
  $\verts{H}-1\geq\chi(G)\geq\chi(H)$.  Then $$\chi(\CP{G}{H})\leq
  \eta(\CP{G}{H})-1\enspace.$$
\end{theorem}

\begin{proof}
  By Sabidussi's Theorem \citep{Sabidussi57} it suffices to prove that
  $\eta(\CP{G}{H})\geq \chi(G)+1$.

  \textbf{Case 1.} $G=K_n$ for some $n\geq3$: Then by
  \propref{CompleteEdge}, $$\eta(\CP{G}{H})\geq\eta(\CP{K_n}{K_2})=n+1=\chi(G)+1=\chi(\CP{G}{H})+1\enspace.$$

  \textbf{Case 2.} $G=C_n$ for some $n\geq3$: Then by
  \propref{CompleteEdge}, $$\eta(\CP{G}{H})\geq\eta(\CP{K_3}{K_2})=4\geq\chi(G)+1=\chi(\CP{G}{H})+1\enspace.$$

  \textbf{Case 3. } $G$ is neither a complete graph nor a cycle: Then
  by Brooks' Theorem \citep{Brooks41} and \corref{GraphStarMinor},
  \begin{align*}
    \eta(\CP{G}{H}) \geq\min\{\verts{H},\Delta(G)+1\}
    \geq\;&\min\{\verts{H},\chi(G)+1\}\\
    =\;&\chi(G)+1 \\
    =\;&\chi(\CP{G}{H})+1\enspace,
  \end{align*}
  as desired.
\end{proof}


Note that \thmref{HadwigerBigH} is best possible, since
\thmref{TreeComplete} implies that for $G=K_n$ and $H$ any tree (no
matter how big), $\chi(\CP{G}{H})=n=\eta(\CP{G}{H})-1$.

\twothmref{HadwigerNearChi}{HadwigerBigH} both prove (under certain
assumptions) that $\chi(\CP{G}{H})\leq\eta(\CP{G}{H})-1$, which is
stronger than Hadwiger's Conjecture for general graphs. This should
not be a great surprise, since if Hadwiger's Conjecture holds for all
graphs, then the same improved result holds for all non-trivial
products \CP{G}{H} with
$\chi(G)\geq\chi(H)$: $$\chi(\CP{G}{H})=\chi(G)\leq\eta(G)\leq\eta(\CP{G}{K_2})-1\leq\eta(\CP{G}{H})-1\enspace.$$

Whether Hadwiger's Conjecture holds for all non-trivial products
reduces to the following particular case.

\begin{theorem}
  \thmlabel{HadwigerConjecture} Let $G$ be a graph. Then Hadwiger's
  Conjecture holds for every non-trivial product \CP{G}{H} with
  $\chi(G)\geq\chi(H)$ if and only if Hadwiger's Conjecture holds for
  \CP{G}{K_2}.
\end{theorem}

\begin{proof}
  The forward direction is immediate. Suppose that Hadwiger's
  Conjecture holds for \CP{G}{K_2}; that is,
  $\chi(\CP{G}{K_2})\leq\eta(\CP{G}{K_2})$.  Let $H$ be a graph with
  at least one edge and $\chi(G)\geq\chi(H)$.  Then
  $\chi(\CP{G}{H})=\chi(G)=\chi(\CP{G}{K_2})$ by Sabidussi's Theorem
  \citep{Sabidussi57}.  Since $K_2$ is a subgraph of $H$,
  $\eta(\CP{G}{H})\geq\eta(\CP{G}{K_2})$. In summary,
$$\chi(\CP{G}{H})
=\chi(G)=\chi(\CP{G}{K_2})\leq\eta(\CP{G}{K_2})\leq\eta(\CP{G}{H})\enspace.$$
Hence Hadwiger's Conjecture holds for \CP{G}{H}.
\end{proof}

\thmref{HadwigerConjecture} motivates studying $\eta(\CP{G}{K_2})$ in
more detail. By \eqnref{TreewdithTied}, $\eta(\CP{G}{K_2})$ is tied to
\tw{G}, the treewidth of $G$. By a minimum-degree-greedy algorithm,
$\chi(G)\leq\tw{G}+1$. Thus it is tempting to conjecture that the
lower bound on $\eta(\CP{G}{K_2})$ from \lemref{LowerBoundGraphEdge}
can be strengthened to
\begin{equation}
  \eqnlabel{StrongTreewidthBound}
  \eta(\CP{G}{K_2})\geq\tw{G}+1\enspace.
\end{equation}
This would imply that for all graphs $G$ and $H$ both with at least
one edge and $\chi(G)\geq\chi(H)$,
$$\chi(\CP{G}{H})=\chi(G)\leq\tw{G}+1\leq\eta(\CP{G}{K_2})\leq\eta(\CP{G}{H})\enspace;$$
that is, Hadwiger's Conjecture holds for every non-trivial product.
However, \eqnref{StrongTreewidthBound} is false. \citet{KB92} proved
that a random cubic graph on $n$ vertices has $\tw{G}\geq\Omega(n)$
but $\eta(\CP{G}{K_2})\leq\Oh{\sqrt{n}}$ by \lemref{UpperBound}.

We finish with some comments about Hadwiger's Conjecture for
$d$-dimensional products. In what follows $G_1,\dots,G_d$ are graphs,
each with at least one edge, such that
$\chi(G_1)\geq\dots\geq\chi(G_d)$. Thus
$\chi(\CCP{G_1}{G_2}{G_d})=\chi(G_1)$ by Sabidussi's Theorem
\citep{Sabidussi57}. Observe that \thmref{HadwigerConjecture}
generalises as follows: Hadwiger's Conjecture holds for all
\CCP{G_1}{G_2}{G_d} if and only if it holds for
\CP{G_1}{Q_{d-1}}. (Recall that $Q_d$ is the $d$-dimensional
hypercube.)\
Finally we show that if Hadwiger's Conjecture holds for all graphs,
then a significantly stronger result holds for $d$-dimensional
products. By \eqnref{HypercubeLower} and \thmref{CompleteComplete},
\begin{align*}
    \eta(\CCP{G_1}{G_2}{G_d}) 
  \geq\;&
  \eta(\CP{K_{\eta(G_1)}}{Q_{d-1}})\\
  \geq\;&
  \eta(\CP{K_{\eta(G_1)}}{K_{2^{d/2}}})\\
  \geq\;&
(2^{d/4}-o(1))  \,\eta(G_1)\\
  \geq\;&
(2^{d/4}-o(1))  \,\chi(G_1)\\
  =\;& 
(2^{d/4}-o(1))\,\chi(\CCP{G_1}{G_2}{G_d})
\enspace.
\end{align*}
This shows that if Hadwiger's Conjecture holds for all graphs, then
the multiplicative factor of $1$ in Hadwiger's Conjecture can be
improved to an exponential in $d$ for $d$-dimensional products.

\section*{Note Added in Proof}

\lemref{LowerBoundGraphEdge} can be restated as: if $\tw{G}\geq
2^{4\ell^4}$ then $\eta(\CP{G}{K_2})\geq\ell$. This exponential bound
was recently improved by \citet{ReedWood-EuJC} to the
following polynomial bound: for some constant $c$, if $\tw{G}\geq
c\ell^4\sqrt{\log\ell}$ then $\eta(\CP{G}{K_2})\geq\ell$.
Subsequently, the bounds in \eqnref{TreewdithTied}, in
\thmref{HadwigerBigTW}, and in the proof of
\lemref{BoundedHadwigerImplies} can be improved.

\end{document}